\documentclass{amsart}
\usepackage[english]{babel}
\usepackage{amssymb,amsmath,amsfonts,amsthm,epsfig,amscd}
\usepackage[dvipsnames]{xcolor}
\usepackage[bookmarksopen,bookmarksdepth=2]{hyperref}
\hypersetup{colorlinks=true,citecolor=NavyBlue,linkcolor=BrickRed,urlcolor=Green} 
\usepackage[font=footnotesize,labelfont=bf]{caption}
\usepackage[mathscr]{eucal}
\usepackage{mathrsfs,calligra}
\usepackage{tikz}
\usepackage{tikz-cd}
\usepackage{tikzsymbols} 
\usepackage{adjustbox}
\usepackage{graphicx}
\usepackage{lmodern}
\usepackage{multirow}
\usepackage{enumitem}
\usepackage{verbatim}
\usepackage{mathtools}
\usepackage{fancyhdr}
\usepackage{threeparttable}
\usepackage{microtype}
\usepackage{booktabs,array}

\usepackage[all,cmtip,poly]{xy}
\usepackage{relsize}

\usepackage{float}
\newcolumntype{C}{>{$}c<{$}} 

\newcounter{rownumber}[table]
\renewcommand{\therownumber}{(\roman{rownumber})}
\newcolumntype{N}{>{\refstepcounter{rownumber}\therownumber}c}

\DeclareUnicodeCharacter{2212}{-}

\theoremstyle{plain}

\newtheorem{thm}[equation]{Theorem}
\newtheorem{prop}[equation]{Proposition}

\newtheorem{lemma}[equation]{Lemma}
\newtheorem{cor}[equation]{Corollary}

\theoremstyle{definition}

\newtheorem{remark}[equation]{Remark}

\newtheorem{defn}[equation]{Definition}

\numberwithin{equation}{section}
\numberwithin{figure}{subsection}

\foreach \theoremenv in {theorem, thm, prop, proposition,lemma, lem, remark, cor, corollary, sublem, fact, conj, rem, notn, ex, example, defn, assumption, exercise, definition} {
  \AtBeginEnvironment{\theoremenv}{%
    \setlist[enumerate,1]{label={(\roman*)}}
  }
}

\newif\iffinalrun
\iffinalrun
\else
\fi

\iffinalrun
  \newcommand{\need}[1]{}
  \newcommand{\mar}[1]{}
\else
  \newcommand{\need}[1]{{\tiny *** #1}}
  \newcommand{\mar}[1]{\marginpar{\raggedright\tiny fixme #1}}
\fi

\newcommand{\F}{\FF}

\newcommand{\Q}{\QQ}

\newcommand{\Z}{\ZZ}

\newcommand{\FF}{{\mathbf F}}

\newcommand{\QQ}{{\mathbf Q}}

\newcommand{\ZZ}{{\mathbf Z}}

\newcommand{\cL}{{\mathcal L}}

\newcommand{\cP}{{\mathcal P}}

\newcommand{\cR}{{\mathcal R}}

\newcommand{\cX}{{\mathcal X}}
\newcommand{\cY}{{\mathcal Y}}
\newcommand{\cZ}{{\mathcal Z}}

\newcommand{\Fbar}{\overline{\F}}
\newcommand{\Qbar}{\overline{\Q}}

\newcommand{\Fpbar}{\Fbar_p}

\newcommand{\Qp}{\Q_p}

\newcommand{\Qpbar}{\Qbar_p}

\DeclareMathOperator{\GL}{GL}

\DeclareMathOperator{\sheafHom}{\mathcal{H}\textit{\kern -1pt{om}}\,}
\DeclareMathOperator{\relSpec}{\mathcal{S}\textit{\kern -1pt{pec}}\,}

\DeclareMathOperator{\Sym}{Sym}

\newcommand{\St}{\mathrm{St}}

\newcommand{\ur}{\mathrm{ur}}

\newcommand{\rhobar}{\overline{\rho}}

\newcommand{\into}{\hookrightarrow}

\newcommand{\BT}{\operatorname{BT}}
\newcommand{\red}{\operatorname{red}}

\newcommand{\ru}{{\underline{r}}}

\newcommand{\eps}{\varepsilon}

\newcommand{\Zss}{\cZ_{\mathrm{ss}}}
\newcommand{\rZss}{Z_{\mathrm{ss}}}

\newcommand{\sss}{\mathrm{ss}}
\newcommand{\crys}{\mathrm{crys}}
\newcommand{\crit}{\mathrm{cr}}

\DeclareSymbolFont{eulerletters}{U}{eur}{m}{n}
\DeclareMathSymbol{\eulers}{\mathord}{eulerletters}{`s}
\DeclareMathSymbol{\eulert}{\mathord}{eulerletters}{`t}

\usetikzlibrary{decorations.markings}

\tikzset{
    labl/.style={anchor=south, rotate=90, inner sep=.5mm}
}

\makeatletter
\tikzcdset{
  open/.code     = {\tikzcdset{hook, circled};},
  closed/.code   = {\tikzcdset{hook, slashed};},
  circled/.code  = {\tikzcdset{markwith = {\draw (0,0) circle (.5ex);}};},
  slashed/.code  = {\tikzcdset{markwith = {\draw[-] (-.4ex,-.4ex) -- (.4ex,.4ex);}};},
  markwith/.code ={
    \pgfutil@ifundefined%
    {tikz@library@decorations.markings@loaded}%
    {\pgfutil@packageerror{tikz-cd}{You need to say %
      \string\usetikzlibrary{decorations.markings} to use arrows with markings}{}}{}%
    \pgfkeysalso{/tikz/postaction = {
      /tikz/decorate,
      /tikz/decoration={markings, mark = at position 0.4 with {#1}}}
    }
  },
}
\makeatother
\title{Inclusions between $p$-bounded crystalline loci in dimension two}
\author{Kalyani Kansal}
\author{Brandon Levin}
\author{David Savitt}
\begin{document}

\begin{abstract}
Let $p$ be an odd prime and $K/\Qp$ a finite unramified extension of degree $f>1$. Let $\cZ(\ru)$ be the reduced special fiber of the Emerton--Gee stack of two-dimensional crystalline representations of $\mathrm{Gal}(\overline{K}/K)$ of Hodge type $\ru$. We study the collection of stacks $\cZ(\ru)$ as $\ru$ varies over $p$-bounded Hodge types, as a set partially ordered under inclusion. We prove that aside from two degenerate cases, simple inclusions can be classified in terms of three operations on Hodge types, two of which have standard automorphic interpretations. We also prove, with one exception, that inclusions can be detected at the level of an inclusion of closed points (equivalently, semisimple mod $p$ Galois representations).  GPT-5.5 Pro was used extensively in the course of this work.

\end{abstract}
\maketitle

\tableofcontents

\section{Introduction}

\subsection{Main results and reductions} Let $p$ be a prime number and $K/\Qp$ a finite extension with residue field $k$. Emerton and Gee \cite{EGmoduli} have constructed a moduli stack $\cX_2$ of rank $2$ $(\varphi,\Gamma)$-modules over $K$ that can be conceived of as a moduli stack of rank $2$ representations of $G_K$, the absolute Galois group of~$K$.   Write $\cX_{2,\red}$ for the underlying reduced substack of $\cX_2$;\ it is an algebraic stack over $\Fpbar$ of dimension $[K:\Qp]$.  The stack $\cX_2$ has deep connections to the representation theory of $\GL_2(K)$. For example, by  \cite[Thm.~1.2.1]{EGmoduli} the irreducible components of $\cX_{2,\red}$ are in bijection with \emph{Serre weights}, the isomorphism classes of irreducible $\Fpbar$-representations of $\GL_2(k)$.

Assume henceforth that $p$ is odd and $K/\Qp$ is unramified of degree $f > 1$, and set $q = p^f$.

This note is a complement to the forthcoming paper \cite{KLSlgc}. Its purpose is to describe, as a set partially ordered under inclusion, a certain collection of closed substacks of $\cX_{2,\red}$ satisfying a natural $p$-adic Hodge-theoretic condition. To explain precisely which closed substacks  we will consider, we first recall what we mean by a Hodge type.  Since $K/\Qp$ is unramified we can (and do) identify the embeddings $K \into \Qpbar$ with embeddings $k \into \Fpbar$. Fix an embedding $\kappa_0 : k \into \Fpbar$ and recursively label the remaining embeddings by elements of $\Z/f\Z$ by taking $\kappa_{i+1}^p = \kappa_i$.

\begin{defn}\label{def:hodge-type}
A \emph{Hodge type} of rank $2$ is a tuple of integers $ \ru = \{r_{i,j}\}_{\kappa_i: K \into \Qpbar, 1 \le j \le 2}$ with $r_{i,1} \ge r_{i,2}$ for all $i$.
We say that the Hodge type~$\ru$ is 
\begin{itemize}
\item \textit{$p$-bounded} if $r_{i,1} - r_{i,2} \leq p$ for all $i$, 
\item \textit{regular} if $r_{i,1} - r_{i,2} > 0$ for all $i$. \end{itemize} 
If $\ru$ is a Hodge type and $\lambda \in \Z^{\Z/f\Z}$, we let $\ru + \lambda$ denote the translation of $\ru$ by $\lambda$, i.e., the Hodge type $(r_{i,1} + \lambda_i, r_{i,2} + \lambda_i)$.

Write $\ur$, $\BT$, $\crit$, and $\St$  for the Hodge types with $(r_{i,1},r_{i,2}) = (r,0)$ for all $i$, for $r = 0,1,p-1,p$ respectively. These stand for unramified, Barsotti--Tate, critical, and Steinberg. We say that a Hodge type is  \emph{scalar}, \emph{Barsotti--Tate}, \emph{critical}, or \emph{Steinberg} if it is the translation of $\ur$, $\BT$, $\crit$, or $\St$ by some $\lambda$.
\end{defn}

For each Hodge type $\ru$ of rank $2$, there is a closed substack $\cX^{\crys,\ru}_2 \subset \cX_2$ corresponding to crystalline representations with $\kappa_i$-labeled Hodge--Tate  weights $\{r_{i,1},r_{i,2}\}$ for each $i$. Its underlying reduced substack $\cZ(\ru) := \cX^{\crys,\ru}_{2,\red}$ is equidimensional of codimension $\#\{i : r_{i,1} = r_{i,2} \}$. If $\Lambda \subset \Z^{\Z/f\Z}$ is the sublattice consisting of tuples $(\lambda_i)$ such that $\sum_{i} p^{-i} \lambda_i \equiv 0 \pmod{q-1}$, then it is easy to see that $\cZ(\ru) = \cZ(\ru + \lambda)$ for any $\ru$ and $\lambda \in \Lambda$. Indeed, on Galois representations this amounts to twisting by a crystalline character whose Hodge--Tate weights are $(\lambda_i)_i$ and whose reduction mod $p$ is trivial.

If $\ru$ is regular, $p$-bounded, and non-Steinberg, then $\cZ(\ru)$ is a single irreducible component of $\cX_{2,\red}$. In our normalizations (following the normalizations of \cite{cegsA,bellovin2024irregular,KLSlgc}) it is the component $\cX(\sigma(\ru))$ corresponding to the Serre weight 
\[ \sigma(\ru) := \otimes_{i=0}^{f-1} (\det {}^{1 - r_{i,1}} \Sym^{r_{i,1} - r_{i,2} - 1} k^2) \otimes_{k,\kappa_i} \Fpbar. \]
The map $\ru \mapsto \sigma(\ru)$ is many-to-one:\ indeed, its fibers are equivalence classes under translation by $\Lambda$.

If instead $\ru$ is Steinberg,  it is proved in \cite{EGmoduli,cegsA} that $\cZ(\ru)$  is the union of two irreducible components. To be precise, if $\ru = \St + \lambda$ then one of those components is $\cZ(\BT + \lambda)$, while by \cite[Lem.~8.6.4]{EGmoduli} all the $\Fpbar$-points of the other component are twists of an extension of the trivial character by the cyclotomic character. This latter component is $\cX(\sigma(\ru))$ where $\sigma(\ru)$ is the Steinberg Serre weight $ \otimes_{i=0}^{f-1} (\det {}^{1 - r_{i,1}} \Sym^{p-1} k^2) \otimes_{k,\kappa_i} \Fpbar$,  so that $\cZ(\St + \lambda) = \cZ(\BT + \lambda) \cup \cX(\sigma(\ru))$.

As $\ru$ ranges over all $p$-bounded and regular Hodge types, $\cX(\sigma(\ru))$ ranges over the irreducible components of $\cX_{2,\red}$, while $\cZ(\ru)$ ranges over the \emph{Breuil--M\'ezard cycles} for $\GL_2(K)$ as in \cite[Conj.~1.7.2]{EGmoduli} and the discussion following it. This means that $\cZ(\ru)$ is supported at a representation $\rhobar : G_K \to \GL_2(\Fpbar)$ if and only if $\sigma(\ru)$ is a member of the set of Serre weights associated to $\rhobar$ in the weight part of Serre's conjecture. 

In this article we will be concerned with the closed substacks $\cZ(\ru)$ as $\ru$ ranges over all $p$-bounded (but not necessarily regular) Hodge types. Irregular Hodge types can be thought of as corresponding to ``phantom Serre weights'' in which some or all of the symmetric powers are allowed to be $-1$. The loci $\cZ(\ru)$ for irregular $\ru$ play a role in the local study of Hilbert modular forms of partial weight one, but as these stacks have positive codimension in $\cX_{2,\red}$, they are somewhat more difficult to study than the irreducible components. However, it is proved in \cite{bellovin2024irregular} that if~$\ru$ is $p$-bounded and non-Steinberg then $\cZ(\ru)$ is irreducible.

\begin{defn}
    Let $Z$ be the set of closed substacks $\cZ(\ru)$ as $\ru$ ranges over all $p$-bounded Hodge types.
\end{defn}

The set $Z$ is partially ordered under inclusion (i.e.\ under closed immersion), and our specific goal in this article is to describe this structure. To state our main results, we need to introduce several operations on Hodge types. 

\begin{defn}\label{def:operators}
    Let $f\geq 2$. For each $j \in \Z/f\Z$, we define operators $\mu_j$, $\theta_j$, and $\nu_j$ on $p$-bounded Hodge types $\underline{r}$
   by setting:
           \begin{align*}
    \mu_{j}(\underline{r})_{{i}} &:= \left\{ \begin{array}{ll}
        (r_{{i},1}-1,r_{{i},2}) & \text{if } {i}= {j}, \\
         (r_{i,1}+p,r_{i,2}) & \text{if } {i}= {j+1}, \\
         (r_{i,1},r_{i,2}) & \text{otherwise. }  
         \end{array} \right.  \end{align*}
         
\begin{align*}
         \theta_{j}(\underline{r})_{{i}} &:= \left\{ \begin{array}{ll}
        (r_{{i},1},r_{{i},2}-1) & \text{if } {i}= {j}, \\
         (r_{{i},1}+p,r_{{i},2}) & \text{if } {i}= {j+1}, \\
         (r_{{i},1},r_{{i},2}) & \text{otherwise. }  
         \end{array} \right.
         \end{align*}
 
         \begin{align*}
        \nu_{j}(\underline{r})_{{i}} &:= \left\{ \begin{array}{ll}
        (r_{i,1},r_{i,2}-1) & \text{if } {i}= {j}, \\
         (r_{i,2} + p,r_{i,1}) & \text{if } {i}= {j}+1, \\
         (r_{i,1},r_{i,2}) & \text{otherwise. }  
         \end{array} \right.
\end{align*}
The difference between $\theta_j$ and $\nu_j$ lies in the subscripts at index $i=j+1$. We say that $\ru$ is irregular at $j$ if $r_{j,1} = r_{j,2}$.
The operator $f_j \in \{\mu_j,\theta_j,\nu_j\}$ is \emph{valid} (for $\ru$) if
\begin{itemize}
    \item $f_j = \mu_j$,  $\ru$ is irregular at $j+1$, and $r_{j,1} - r_{j,2} >  0$;

    \item $f_j = \theta_j$, $\ru$ is irregular at $j+1$,  and $r_{j,1} - r_{j,2} < p$;

    \item $f_j = \nu_j$ and $\ru$ is irregular at $j$.
\end{itemize}
If $f_j$ is valid for $\ru$ then evidently $f_j(\ru)$ is again a $p$-bounded Hodge type.
\end{defn}

The definitions of the operators $\mu_j$ and $\theta_j$ are motivated by the theory of mod $p$ automorphic forms:\ the weights of mod $p$ automorphic forms transform by $\mu_j$ under partial Hasse invariants, and by $\theta_j$ under partial theta operators. The operator~$\nu_j$ does not have as straightforward an automorphic explanation,  
but can be related to restriction  to the vanishing locus of a partial Hasse invariant;\ see \cite{SYang2025geometric} or \cite{KLSlgc} for much more. 

\begin{remark}
    We caution the reader that the indexing of the operators $\mu_j$ and $\theta_j$ differs by $1$ from the indexing in \cite{bellovin2024irregular}. We have made this change, both here and in \cite{KLSlgc}, so that $\mu_j$ corresponds to multiplication by the partial Hasse invariant $\textrm{Ha}_j$ rather than $\textrm{Ha}_{j-1}$, and similarly for $\theta_j$ and the partial theta operator $\Theta_j$ rather than $\Theta_{j-1}$. There is an added advantage that $\mu_j$, $\theta_j$, and $\nu_j$ all affect Hodge--Tate weights in the same pair of embeddings;\ the disadvantage is that $\mu_j$ and $\theta_j$ affect an irregular embedding at index $j+1$ rather than $j$. 
\end{remark}

The following proposition is \cite[Cor.~5.3]{bellovin2024irregular}.

\begin{prop}\label{prop:operation-inclusion}
If $\ru$ is a $p$-bounded Hodge type and the operation $f_j \in \{\mu_j,\theta_j,\nu_j\}$ is valid, then $\cZ(\ru) \subseteq \cZ(f_j(\ru))$.    
\end{prop}

To simplify notation,  write $u_i := r_{i,1} - r_{i,2}$.  If $(u_{j},u_{j+1}) = (1,0)$, it is not hard to see that $\nu_{j}(\mu_j(\ru)) \in  \ru + \Lambda$. Similarly if $(u_{j},u_{j+1}) = (0,p)$ then $\mu_{j}(\nu_j(\ru)) \in  \ru + \Lambda$. Thus if $(u_{j},u_{j+1}) = (1,0)$ we have $\cZ(\mu_j(\ru)) = \cZ(\ru)$, while if $(u_{j},u_{j+1}) = (0,p)$ we have $\cZ(\nu_j(\ru)) = \cZ(\ru)$. In these situations we say that the operations $\mu_j$ and $\nu_j$ are invertible. 

\begin{remark}\label{rem:X-vs-Z}
    One could alternatively define a set  $X = \{ \cX(\ru) : \text{$p$-bounded}\ \ru \}$ by setting $\cX(\ru) = \cZ(\ru)$ if $\ru$ is non-Steinberg (equivalently, if $\cZ(\ru)$ is irreducible) and setting $\cX(\ru) = \cX(\sigma(\ru))$ in the Steinberg case. This would seem to have the advantage that all the stacks in $X$ are irreducible. 
    
    The reason we choose to work with $Z$ rather than $X$ is that Proposition~\ref{prop:operation-inclusion} is no longer quite correct if $\cZ$ is replaced with $\cX$. If $\ru$ is the Hodge type such that $\theta_j(\ru) = \St$, then $\cX(\ru) \not\subseteq \cX(\St)$, as can be seen by checking that there are irreducible representations $\rhobar$ in the support of $\cX(\ru)$ but none in that of $\cX(\St)$. In fact if $\theta_{j}(\ru) = \St$ then we have 
\begin{equation}\label{eq:St-to-BT}
   (\nu_{j+f-1} \circ \cdots \circ \nu_{j+2} \circ \nu_{j+1})(\ru) \in \BT + \Lambda,
\end{equation}
where the composition has $f-1$ factors,
 so that indeed $\cZ(\ru) \subseteq \cZ(\St)$, but with image in the component $\cX(\BT)$ rather than $\cX(\St)$. 

 In any case, since $X$ is very nearly the same as $Z$, it is possible to reformulate all of our results in terms of $X$ rather than $Z$ if one prefers.
\end{remark}

\begin{defn}
    Let $\sim$ denote the equivalence relation on $p$-bounded Hodge types generated by $\ru \sim \mu_j(\ru)$ for valid invertible $\mu_j$,   by $\ru \sim \nu_j(\ru)$ for valid invertible $\nu_j$, and by translation by $\Lambda$. 
\end{defn}

Our main results are as follows.

\begin{thm}\label{thm:equivalence-Z}
We have $\cZ(\ru) = \cZ(\ru')$ if and only if $\ru \sim \ru'$.
\end{thm}

If $\cZ \subseteq \cZ'$ is an inclusion of elements of $Z$, let us say that the inclusion is \emph{simple} if $\cZ \subsetneq \cZ'$ and furthermore there does not exist $\cY \in Z$ with $\cZ \subsetneq \cY \subsetneq \cZ'$.

\begin{thm}\label{thm:simple-inclusions}
Suppose $\cZ \subsetneq \cZ'$ is a simple inclusion in $Z$. Then either:
\begin{itemize}
    \item there exists a $p$-bounded Hodge type $\ru$ and a valid operation $f_j \in \{\mu_j,\theta_j,\nu_j\}$ such that $\cZ = \cZ(\ru)$ and $\cZ' = \cZ(f_j(\ru))$,

    \item we have $\cZ = \cZ(\ur + \lambda)$ and $\cZ' = \cZ(\crit + \lambda)$ for some $\lambda$, or

    \item we have $\cZ = \cZ(\BT + \lambda)$ and $\cZ' = \cZ(\St + \lambda)$  for some $\lambda$.
\end{itemize}
\end{thm}

Note that in the second case there is indeed an inclusion
$\cZ(\ur+\lambda)\subseteq \cZ(\crit+\lambda)$:\ unramified two-dimensional representations of $G_K$ have crystalline lifts with all labeled Hodge--Tate weights equal to $(p-1,0)$.

A  consequence of Theorem~\ref{thm:simple-inclusions}, and the reason we became interested in these statements in the first place, is that if one wants to use induction on (co)dimension to prove theorems about the stacks $\cZ(\ru)$, as we do in \cite{KLSlgc}, then the only ``moves'' that one has available in general are the operations $\mu_j$, $\theta_j$, and $\nu_j$. Philosophically this suggests that there is indeed something fundamental about the operation $\nu_j$, alongside $\mu_j$ and $\theta_j$ with their evident automorphic interpretations. 

As a consequence of Theorems~\ref{thm:equivalence-Z} and~\ref{thm:simple-inclusions} we have the following:

\begin{cor}\label{cor:chain}
    If $\cZ \subseteq \cZ'$ is an inclusion of elements of $Z$, then either 
    \begin{itemize}
    \item we have $\cZ = \cZ(\ur + \lambda)$ and $\cZ' = \cZ(\crit + \lambda)$  for some $\lambda$, or
    \item there is a chain 
    \[ \cZ = \cZ_0 \subsetneq \cdots \subsetneq \cZ_n = \cZ'\]
    of simple inclusions for some $n \ge 0$, and there exist $p$-bounded Hodge types $\ru_0,\ldots,\ru_n$ such that $\cZ_m = \cZ(\ru_m)$ for all $0 \le m \le n$, and for every $1 \le m \le n$ we have either:
   \begin{itemize}
       \item $\cZ_m = \cZ(f^{(m)}(\ru_{m-1}))$ for some valid  $f^{(m)}$,  and  $f^{(m)}(\ru_{m-1}) \sim \ru_m$, or 
       \item $m=n$ and the inclusion $\cZ_{n-1} \subsetneq \cZ_n$ is some $\cZ(\BT + \lambda) \subsetneq \cZ(\St + \lambda)$.
       \end{itemize}
   \end{itemize} 
\end{cor}

\begin{proof}
The existence of a chain $\cZ = \cZ_0 \subsetneq \cdots \subsetneq \cZ_n = \cZ'$ of simple inclusions follows for dimension reasons. If one of these inclusions has the form $\cZ(\ur + \lambda) \subsetneq \cZ(\crit + \lambda)$, then $n=1$:\ for dimension reasons there are no proper inclusions into $\cZ(\ur + \lambda)$  or out of 
$\cZ(\crit + \lambda)$. Similarly any inclusion in the chain that is of the form $\cZ(\BT + \lambda) \subsetneq \cZ(\St + \lambda)$ must be last in the chain. Then Theorems~\ref{thm:equivalence-Z} and~\ref{thm:simple-inclusions} imply that the chain has the asserted form.
\end{proof}

It remains to describe precisely which non-invertible valid operations give simple inclusions;\ or equivalently, which ones do not. Note that on dimension grounds, the latter would require either reducing the number of irregular embeddings by at least two, or increasing the number of irreducible components. This can  occur only if $f_j(\ru)$ is Steinberg, or if $(u_{j},u_{j+1}) = (0,0)$ and $f_j = \theta_j$ or $\nu_j$. (Note that in the latter situation $\theta_j$ and $\nu_j$ are equal.) The actual behavior is as follows.

\begin{thm}\label{thm:non-simple-inclusions}
Each of the inclusions listed in Theorem~\ref{thm:simple-inclusions} is simple except for the ones of the form 
$\cZ(\ru) \subseteq \cZ(f_j(\ru))$ with either $f_j$ invertible (in which case the inclusion is an equality) or else of the form
\begin{itemize}
    \item $f_j \in \{ \theta_j, \nu_j\}$, $(u_{j},u_{j+1}) = (0,0)$, and either $u_{j-1} = 1$ or $u_{j+2} = p$; or
    \item $f_j = \theta_j$ and $\theta_j(\ru)$ is Steinberg,
\end{itemize}
in which case the inclusion is proper but not simple.
\end{thm}
In the final case, if $\theta_j(\ru) = \St + \lambda$ we have already seen that $\cZ(\ru) \subsetneq \cZ(\BT+\lambda) \subsetneq \cZ(\St + \lambda)$. In the first case, if for example $u_{j-1} = 1$, then $\theta_j(\ru) \sim \nu_{j-1}(\mu_{j}(\mu_{j-1}(\ru)))$, and $\cZ(\mu_{j}(\mu_{j-1}(\ru)))$ is intermediate between $\cZ(\ru)$ and $\cZ(\theta_j(\ru))$.  The other cases of non-simplicity are similar.

It is possible to reduce each of Theorems~\ref{thm:equivalence-Z},~\ref{thm:simple-inclusions}, and~\ref{thm:non-simple-inclusions}   to statements that are essentially combinatorial, as we now explain.

\begin{defn}
If $\ru$ is a Hodge type, define $\Zss(\ru)$ to be the set of semisimple representations $\rhobar : G_K \to \GL_2(\Fpbar)$ on which $\cZ(\ru)$ is supported, i.e., the set of semisimple~$\rhobar$ having a crystalline lift of Hodge type $\ru$;\ equivalently, $\Zss(\ru)$ is the collection of closed points of $\cZ(\ru)$. 

Write $\rZss$ for the set  $\{ \Zss(\ru) : \text{$p$-bounded}\ \ru \}$.
\end{defn}

Since the semisimple representations at which $\cX(\St)$ is supported also lie in the support of $\cX(\BT)$, we have $\Zss(\BT + \lambda) = \Zss(\St + \lambda)$.  The main combinatorial claim is that versions of Theorems~\ref{thm:equivalence-Z},~\ref{thm:simple-inclusions}, and~\ref{thm:non-simple-inclusions} hold  with $\cZ$ replaced by $\Zss$, suitably modified to account for the equality $\Zss(\BT + \lambda) = \Zss(\St + \lambda)$.

\begin{thm}\label{thm:Zss}
   We have the following.
 \begin{enumerate}
     \item   We have $\Zss(\ru) = \Zss(\ru')$ if and only if  either $\ru \sim \ru'$ or else there exists $\lambda$ such that either $(\ru \sim \BT + \lambda$ and $\ru' \sim \St + \lambda)$ or $(\ru' \sim \BT + \lambda$ and $\ru \sim \St + \lambda)$. 

\item Suppose $\Zss \subsetneq \Zss'$ is a simple inclusion in $\rZss$. Then either:
\begin{itemize}
    \item there exists a $p$-bounded Hodge type $\ru$ and a valid operation $f_j \in \{\mu_j,\theta_j,\nu_j\}$ such that $\Zss = \Zss(\ru)$ and $\Zss' = \Zss(f_j(\ru))$, or

    \item we have $\Zss = \Zss(\ur + \lambda)$ and $\Zss' = \Zss(\crit + \lambda)$ for some $\lambda$.
\end{itemize}

\item Each of the inclusions listed in $(ii)$ is simple except for the ones of the form 
$\Zss(\ru) \subsetneq \Zss(f_j(\ru))$ with $f_j \in \{ \theta_j, \nu_j\}$, $(u_{j},u_{j+1}) = (0,0)$ and either $u_{j-1} = 1$ or $u_{j+2} = p$.
\end{enumerate}
\end{thm}

For $p$-bounded Hodge types~$\ru$, the  results of \cite{gls12} give an explicit description of $\Zss(\ru)$, thus reducing Theorem~\ref{thm:Zss} to an explicit but very complicated combinatorial problem whose proof will occupy the body of the paper. For now, however, we prove the following.

\begin{prop}\label{prop:reduction}
  Assume parts $(i)$ and $(ii)$ of Theorem~\ref{thm:Zss}. Then we have:
\begin{enumerate}
    \item If $\ru$ and $\ru'$ are $p$-bounded Hodge types, then $\cZ(\ru) \subseteq \cZ(\ru')$ if and only if $\Zss(\ru) \subseteq \Zss(\ru')$, except when there exists $\lambda$ such that $\ru \sim \St + \lambda$ and $\ru' \sim \BT + \lambda$.
    
    \item Theorems~\ref{thm:equivalence-Z} and~\ref{thm:simple-inclusions} hold.

    \item Further assuming Theorem~\ref{thm:Zss}(iii), Theorem~\ref{thm:non-simple-inclusions} holds.
\end{enumerate}
\end{prop}

\begin{remark}
   We emphasize the following interpretation of   Proposition~\ref{prop:reduction}$(i)$:\ for $p$-bounded Hodge types $\ru, \ru'$, inclusions between the stacks $\cZ(\ru)$ and $\cZ(\ru')$ may be detected at the level of an inclusion between their closed points, with the exception that $\Zss(\St + \lambda) = \Zss(\BT+\lambda)$ while $\cZ(\St+\lambda)\not\subseteq \cZ(\BT + \lambda)$.
\end{remark}

 \begin{proof} 
 The only-if direction of $(i)$ is immediate. For the reverse, assume that $\Zss(\ru) \subseteq \Zss(\ru')$,
 and that there does not exist $\lambda$ such that $\ru \sim \St + \lambda$ and $\ru' \sim \BT + \lambda$.

 If $\Zss(\ru) = \Zss(\ru')$, then Theorem~\ref{thm:Zss}$(i)$ implies that either $\ru \sim \ru'$ or else $\ru \sim \BT + \lambda$ and $\ru' \sim \St + \lambda$ for some $\lambda$. In either case $\cZ(\ru) \subseteq \cZ(\ru')$.
 
 Otherwise, since the set $\rZss$ is finite it is possible to write 
\[ \Zss(\ru) = \cZ_{\sss,0} \subsetneq \cZ_{\sss,1} \subsetneq \cdots \subsetneq \cZ_{\sss,n} = \Zss(\ru') \] for some $n > 0$,  with each $\cZ_{\sss,m} \subsetneq \cZ_{\sss,m+1}$ a simple inclusion in $\rZss$. By Theorem~\ref{thm:Zss}$(ii)$, we can write this inclusion as $\Zss(\ru_m) \subsetneq \Zss(\ru'_m)$ for  Hodge types $\ru_m, \ru'_m$ satisfying one of the two listed possibilities. 
In each of those cases we know that $\cZ(\ru_m) \subseteq \cZ(\ru'_m)$. 

Regular Hodge types give maximal elements of $\rZss$ (e.g.\ by Theorem~\ref{thm:Zss}$(ii)$, because~$\ru$ is never regular in the first  bullet point, and similarly $\ur + \lambda$ is not regular in the second bullet point). Thus none of $\ru$, $\ru_m$, or $\ru'_m$ is either Barsotti--Tate or Steinberg except possibly $\ru'_{n-1}$.  By Theorem~\ref{thm:Zss}$(i)$, $\Zss(\ru'_m) = \Zss(\ru_{m+1})$ implies that $\ru'_m \sim \ru_{m+1}$, and therefore $\cZ(\ru'_m) = \cZ(\ru_{m+1})$. Similarly $\cZ(\ru) = \cZ(\ru_0)$. If $\ru'_{n-1}$ (hence also $\ru'$) is either Barsotti--Tate or Steinberg, then by changing $\ru_{n-1}$ it is possible to arrange that $\ru'_{n-1} \sim \ru'$:\ this is because if $\theta_j(\underline{s}) = \St + \lambda$, then  there exists $\underline{s}' \sim \underline{s}$ such that $\nu_{j+f-1}(\underline{s}') \sim \BT + \lambda$ (namely $\underline{s}' \sim (\nu_{j+f-2} \circ \cdots \circ \nu_{j+2} \circ \nu_{j+1})(\underline{s})$, as in \eqref{eq:St-to-BT}). We obtain 
\[ \cZ(\ru) = \cZ(\ru_0) \subseteq \cZ(\ru'_0) = \cZ(\ru_1) \subseteq \cdots \subseteq \cZ(\ru'_{n-1}) = \cZ(\ru') \]
as desired.

Theorem~\ref{thm:Zss}$(i)$ implies that there exists a well-defined map from $\rZss$ to $Z \smallsetminus \{ \cZ(\St + \lambda) : \lambda \in \Z^{\Z/f\Z}\}$ sending $\Zss(\ru)$ to $\cZ(\ru)$, except that $\Zss(\St + \lambda) = \Zss(\BT + \lambda)$ is sent to $\cZ(\BT + \lambda)$. 

By $(i)$ of this Proposition, this is an isomorphism of partially ordered sets under inclusion. By the same argument with $\St$ and $\BT$ swapped, there is also a well-defined isomorphism from $\rZss$ to $Z \smallsetminus \{ \cZ(\BT + \lambda) : \lambda \in \Z^{\Z/f\Z}\}$. This establishes Theorems~\ref{thm:equivalence-Z},~\ref{thm:simple-inclusions}, and~\ref{thm:non-simple-inclusions} for all statements not involving both Barsotti--Tate and Steinberg weights (under the additional assumption of Theorem~\ref{thm:Zss}$(iii)$, in the case of Theorem~\ref{thm:non-simple-inclusions}).

The claims involving both Barsotti--Tate and Steinberg weights are easily checked separately. We have $\cZ(\BT + \lambda) \neq \cZ(\St + \lambda')$ and $\BT + \lambda \not\sim \St + \lambda'$, giving Theorem~\ref{thm:equivalence-Z}. The inclusion $\cZ(\BT + \lambda) \subseteq \cZ(\St + \lambda)$ is indeed simple, necessitating the third bullet point in Theorem~\ref{thm:simple-inclusions}, and there are no  inclusions between $\cZ(\BT + \lambda)$ and $\cZ(\St + \lambda')$ if $\lambda - \lambda' \not\in \Lambda$.
For Theorem~\ref{thm:non-simple-inclusions}, a statement involving both Barsotti--Tate and Steinberg weights could only have $\St + \lambda = f_j(\ru)$, with the Barsotti--Tate weight appearing in the middle of a non-simple inclusion;\ and indeed we have already seen that if $\theta_j(\ru) = \St + \lambda$, then $\cZ(\ru) \subsetneq \cZ(\BT+ \lambda) \subsetneq \cZ(\St + \lambda)$, giving the second bullet point of Theorem~\ref{thm:non-simple-inclusions} and completing the proof.
\end{proof}

The rest of the paper is devoted to the proof of Theorem~\ref{thm:Zss}.

\subsection{Statement of AI use} Given that for a valid operation $f_j$, the inclusion $\cZ(\ru) \subseteq \cZ(f_j(\ru))$ is typically a codimension one inclusion between irreducible stacks, it was natural to wonder whether something like Theorem~\ref{thm:simple-inclusions} might hold. In the course of our forthcoming work in \cite{KLSlgc}, we had previously used calculations with sets of semisimple points to make guesses about the extent to which formulas like $\cZ(\mu_j(\ru)) \cap \cZ(\theta_j(\ru)) = \cZ(\ru)$ might hold. This prompted us to consider the analogue of Theorem~\ref{thm:simple-inclusions} for semisimple point sets and to realize that the one could be reduced to the other. It appeared to us that the combinatorics required to prove Theorem~\ref{thm:Zss} might be rather complicated, but also potentially amenable to current large language models.

Indeed the first proof of Theorem~\ref{thm:Zss}  was found by GPT-5.5 Pro, by a naive process of giving the problem to the LLM and then prompting it to continue working whenever it stopped. After roughly 30 continuations, the LLM claimed to have a proof. The argument was then turned into a manuscript by an iterative process using two LLM conversations, one the ``author'' and the other an adversarial ``reviewer''. The resulting manuscript was, to us, essentially alien --- full of jargon and barely recognizable as related to the original problem. Attempts to formalize the manuscript in Lean using Codex eventually stalled.

We then abandoned the initial manuscript and restarted the process from the beginning, this time also prompting the LLM that it was permitted to introduce new notation, but not new terminology. The new manuscript, which ultimately grew to 45 pages, was clearly recognizable as addressing the problem under consideration, and seemed plausibly to be a proof, although the writing  was still very difficult for us to follow. We then formalized the manuscript in Lean via Codex using GPT-5.5 with xhigh reasoning effort. This surfaced several errors in the manuscript, which were corrected with input from GPT-5.5 Pro;\ but the formalization was eventually completed, convincing us that indeed we had a proof.

It then fell to us to understand the argument. It was immediately clear that the manuscript was organized in what seemed to us to be an unnatural and unnecessarily complicated way. In the process of reorganizing the argument, and trying to understand the reorganized argument, we found a number of simplifications to the statements, the proofs, and the exposition.  The exposition below is entirely written by us. Many of the lemma and theorem statements are different from those that can be found in the LLM-generated manuscript. Nevertheless, the underlying logic is still in essence the one that was found by the LLM.

The artifacts that we have described here, including  several versions of LLM-generated manuscripts, and the Lean formalization of one of those early versions, are available at \url{https://github.com/davidsavitt/KLSposet-artifacts}. A formalization of the combinatorial parts of the present manuscript (i.e.\ the results about the sets $\Zss(\ru)$, but not those concerning the Emerton--Gee stacks) is available at \url{https://github.com/davidsavitt/KLSposet}.

\subsection{Notation and further preliminaries}\label{ss:notation}

\subsubsection{Phantom Serre weights}
To each Hodge type $\ru$ we may associate the tuple $(u_0,\ldots,u_{f-1} ; t)$ with $u_i = r_{i,1} - r_{i,2}$ and $t = \sum_{i \in \Z/f\Z} p^{-i} r_{i,2} \in \Z/(q-1)\Z$. This tuple depends only on the class of $\ru$ under translation by $\Lambda$, and it is not difficult to see that this defines a bijection from equivalence classes of $p$-bounded Hodge types up to translation by $\Lambda$, to tuples $(u_0,\ldots,u_{f-1} ; t)$ with $t \in \Z/(q-1)\Z$ and each $u_i \in [0,p]$. We will refer to the tuples $(u_0,\ldots,u_{f-1};t)$ as \emph{phantom Serre weights}.  For brevity we often write $\sigma = (u_i ; t)$ instead of $(u_0,\ldots,u_{f-1};t)$, and we write $(\ur;t)$, $(\BT;t)$, $(\crit;t)$, and $(\St;t)$ for $(u,\ldots,u;t)$ with $u = 0,1,p-1,p$ respectively.

The operations $\mu_j$, $\theta_j$, and $\nu_j$ descend to phantom Serre weights. Concretely, we have
\[
\begin{array}{rcll}
\mu_j:     & u_{j}\mapsto u_{j}-1, & u_{j+1}\mapsto u_{j+1} + p,
       & t\mapsto t,\smallskip\\
\theta_j: & u_{j}\mapsto u_{j}+1, & u_{j+1}\mapsto u_{j+1} + p,
       & t\mapsto t-p^{-j},\smallskip\\
\nu_j:    & u_j\mapsto u_j + 1, & u_{j+1}\mapsto p-u_{j+1},
       & t\mapsto t-p^{-j}+u_{j+1}\,p^{-(j+1)}.
\end{array}
\]
All other $u_i$'s are fixed. The operations are valid when
\begin{itemize}
    \item $u_{j} > 0$ and $u_{j+1} = 0$ for $\mu_j$;
    \item $u_{j} < p$ and $u_{j+1} = 0$ for $\theta_j$; 
    \item $u_j = 0$ for $\nu_j$. 
\end{itemize}
Now if $(u_{j},u_{j+1})=(1,0)$ then $\nu_{j}(\mu_j(\sigma)) = \sigma$, while if $(u_j,u_{j+1}) = (0,p)$ then $\mu_{j}(\nu_j(\sigma)) = \sigma$, so on phantom Serre weights $\mu_j$ and $\nu_j$ are literally inverse to one another.
The equivalence relation $\sim$ also descends to phantom Serre weights, and we denote it by $\sim$ again.

\subsubsection{Fundamental characters}
Take $\pi = (-p)^{1/(q-1)}$ and for $g \in G_K$ we set $h(g) = g(\pi)/\pi \in \mu_{q-1}(K)$. Identifying $\mu_{q-1}(K)$ with $k^{\times}$, we define fundamental characters $\omega_i$ of level $f$ by setting $\omega_i = \kappa_i \circ h : I_K \to \Fpbar^{\times}$  for
     each $i \in \Z/f\Z$ (\emph{cf.}~\cite[Lem.~1.4.1]{cegsC} and the discussion following).  Fix $\omega'_0$ to be either of the two fundamental characters of level $2f$ such that $(\omega'_0)^{q+1} = \omega_0$, and recursively define $\omega'_i$ for $i \in \Z/2f\Z$ by $(\omega'_{i+1})^p = \omega'_i$.

\subsubsection{The sets \texorpdfstring{$\Zss(\sigma)$}{Z\_ss(sigma)}}

We now describe the sets $\Zss(\ru)$ explicitly. 

\begin{defn}\label{def:profile}
We say that a set $J$ is a \emph{profile} if either 
\begin{itemize}
    \item $J \subset \Z/f\Z$ is any subset;\ or
    \item $J \subset \Z/2f\Z$, and for all $i$ we have $i \in J$ if and only if $i+f \not\in J$.
\end{itemize}
In the former case we say that $J$ has niveau $1$, and in the latter case we say that $J$ has niveau $2$.
\end{defn}

If $\sigma = (u_i ; t)$, then for each profile $J \subset \Z/f\Z$ we define
\[ \rhobar(\sigma, J) = \omega_0^t \otimes \left( \prod_{i \in J} \omega_i^{u_i} \oplus \prod_{i \in J^c} \omega_i^{u_i} \right), \]
while for each profile $J \subset \Z/2f\Z$ we define
\[ \rhobar(\sigma, J) = \omega_0^t \otimes \left( \prod_{i \in J} (\omega'_i)^{u_i} \oplus \prod_{i \in J^c} (\omega'_i)^{u_i} \right), \]
with the notation  $u_i$ extended to $i \in \Z/2f\Z$ periodically modulo $f$. 

\begin{thm}[\cite{gls12,cegsA}] If $\sigma$ is the phantom Serre weight associated to $\ru$, then we have 
    \[ \Zss(\ru) = \{ \text{semisimple}\ \rhobar\  \text{such that}\ \rhobar|_{I_K} \cong \rhobar(\sigma,J)^{\vee} \,  \text{for some profile}\ J \}. \]
\end{thm}
The dual comes from our normalization for Hodge--Tate weights:\ following \cite{cegsA,bellovin2024irregular} we chose the normalization so that the cyclotomic character has all Hodge--Tate weights equal to $-1$. 

\begin{defn}
 If $\sigma$ is a phantom Serre weight, we define 
\[ \Zss(\sigma) = \{ \rhobar(\sigma, J) : J\ \text{any profile}\}.\]   
\end{defn}
  Theorem~\ref{thm:Zss} thus descends to an equivalent statement about phantom Serre weights and the finite sets $\Zss(\sigma)$. This is the language in which we will phrase the proof of Theorem~\ref{thm:Zss}.

\begin{remark}\label{rem:same-f}
  Suppose $\rhobar$ is semisimple and non-scalar on inertia. If $J$ is a profile and $\rhobar|_{I_K} \cong \rhobar(\sigma,J)^{\vee}$ then  either $\rhobar$ is split and $J$ has niveau $1$, or else $\rhobar$ is irreducible and $J$ has niveau $2$.

  If on the other hand $\rhobar$ is scalar on inertia, then it can happen that $\rhobar|_{I_K} \cong \rhobar(\sigma,J)^\vee$ with $J$ of niveau 2. But for each such $\sigma$, there is also a witness $\rhobar|_{I_K} \cong \rhobar(\sigma,J')^\vee$ with $J'$ of niveau $1$;\ see the discussion in the last paragraph of \cite[Example~7.1.7]{GHS}.
\end{remark}

If $J$ is a profile, let $\delta_J$ be the indicator function of $J$, and set $\eps^J_i = (-1)^{\delta_J(i)}$.

\section{Pinnings}

\subsection{The carry equation}
The following key lemma is also proved in \cite{KLSlgc}, but we include a proof here for the sake of completeness. 

\begin{lemma}\label{lem:carry-equation}
  Let $\sigma = (u_i;t)$ and $\sigma' = (u'_i;t')$ be two phantom Serre weights. Suppose that $J,J'$ are profiles of the same niveau such that $  \rhobar(\sigma,J) = \rhobar(\sigma',J')$, with  equality rather than isomorphism meaning that the two characters of $ \rhobar(\sigma,J)$ and of $\rhobar(\sigma',J')$ are in the same order.  Set $f' = f$ if $J$ has niveau $1$, and $f' = 2f$ otherwise.

\begin{enumerate}
    \item  There are unique  integers $(x_i)$ such that
  \[ \eps^{J'}_i u'_i - \eps^{J}_i u_i = p x_{i-1} - x_i \]
  for all $i \in \Z/f'\Z$. Moreover $x_{i+f} = -x_i$ if $f' = 2f$. 
  
   \item The parities of the $x_i$'s depend only on $\sigma, \sigma'$ and in particular are independent of the profiles $J,J'$.

   \item We have $x_i \in [-2,2]$ for all $i$, unless $p = 3$, $u_i = u'_i = 3$ for all $i$, $t'-t = (q-1)/2$, and $\{ J, J' \} = \{ \varnothing,\Z/f\Z\}$.
\end{enumerate}
\end{lemma}

The integers $x_i$ in the lemma can be thought of as carries in the base $p$ subtraction $\sum_i \eps_i^{J'} u'_i p^{f'-i} - \sum_i \eps_i^J u_i p^{f'-i}$;\ hence the name of the subsection.

\begin{proof}
  The equality $\rhobar(\sigma,J) = \rhobar(\sigma',J')$ implies that there are congruences
  \begin{equation}\label{eq:rhobar-cong}
  \sum_{i \in J} p^{f'-i} u_i  \equiv T + \sum_{i \in J'} p^{f'-i} u'_i  \pmod{p^{f'}-1}
  \end{equation}
  and 
    \begin{equation}\label{eq:rhobar-cong-2}
  \sum_{i \not\in J} p^{f'-i} u_i  \equiv   T + \sum_{i\not\in J'} p^{f'-i} u'_i  \pmod{p^{f'}-1}
  \end{equation}
where $T = t' - t$ if $f' = f$ and $T = (q+1)(t'-t)$ if $f' = 2f$.

Taking the difference between equations \eqref{eq:rhobar-cong} and \eqref{eq:rhobar-cong-2} and rearranging gives $ \sum_{i=0}^{f'-1} p^{f'-i} y_i \equiv 0 \pmod{p^{f'}-1}$ where  \[ y_i = (-1)^{\delta_{J'}(i)} u'_i - (-1)^{\delta_J(i)} u_i.\]

  The lattice $\Lambda'$ of vectors $(y_i) \in \Z^{\Z/f'\Z}$ such that $ \sum_{i=0}^{f'-1} p^{f'-i} y_i \equiv 0 \pmod{p^{f'}-1}$ has a basis consisting of the vectors $v_j = (v_{j,i})$ with $v_{j,j} = -1$, $v_{j,j+1} = p$, and $v_{j,i} = 0$ otherwise.  It follows immediately that there exist unique integers $x_i$ with $y_i = p x_{i-1} - x_i$. 

   Suppose $f' = 2f$. The sequences $u_i,u'_i$ are periodic modulo $f$, while $i \in J$ if and only if $i +f \not\in J$, and similarly for $J'$. It follows that the tuple $(x'_i)$ with $x'_i := -x_{i+f}$ also satisfies the conditions of the lemma. By uniqueness it follows that $x_i = -x_{i+f}$. This proves $(i)$.

  Let $M = \max_j |x_j|$ and suppose that $|x_{i-1}| = M$. Then 
  \[ pM = |p x_{i-1}| = |y_i + x_i| \le 2p + M.  \]
  If $p > 3$ then $M \le 2p/(p-1) < 3$ by the hypothesis on $p$, and $x_i \in [-2,2]$ for all $i$. 
  
  If $p=3$, then the same conclusion still holds unless $|x_{i-1}| = |x_i| = 3$ and $|y_i| = 6$; then iteratively we obtain $|y_i| = 6$ for all $i$, and $u_i = u'_i = 3$ for all $i$. Finally either $x_i = 3$ for all $i$, in which case $\eps_i^{J'} - \eps_i^J = 2$ for all $i$ and $J = \Z/f\Z$, $J' = \varnothing$;\ or else $x_i = -3$ for all $i$, $\eps_i^{J'} - \eps_i^J = -2$ for all $i$, and $J = \varnothing$, $J' = \Z/f\Z$.
  If \(J=\Z/f\Z\) and \(J'=\varnothing\), then  \eqref{eq:rhobar-cong} gives $T \equiv \sum_{i=0}^{f-1}3^{f-i} \cdot 3 \equiv (q-1)/2 \pmod{q-1}$.
The case \(J=\varnothing\), \(J'=\Z/f\Z\) gives the negative of
the same congruence, hence again \(T=(q-1)/2\).
This gives $(iii)$.
    
  Finally suppose that
  $(K,K')$ is a pair of profiles of the same niveau with $\rhobar(\sigma,K) = \rhobar(\sigma',K')$, leading to parameters $\tilde f'$, $\tilde y_i$ and $\tilde x_i$ in place of $f'$, $y_i$ and $x_i$. Suppose first that $\tilde f' = f'$.   Since $(-1)^{\delta_J(i)} = 1 - 2\delta_J(i)$ the difference $z_i = \tfrac{1}{2}(y_i - \tilde y_i)$ is equal to 
  \[ (\delta_{K'}(i) - \delta_{J'}(i)) u'_i - (\delta_{K}(i) - \delta_J(i)) u_i.\] 
  Taking the difference between \eqref{eq:rhobar-cong} for $(J,J')$ and for $(K,K')$ shows that $(z_i) \in \Lambda'$. Writing $z_i = pw_{i-1} - w_i$ for integers $w_i$ we find that $x_i - \tilde x_i = 2w_i$ and therefore $x_i \equiv \tilde x_i \pmod{2}$.

  If instead $\tilde f' \neq f'$, suppose without loss of generality that $f' =f$ and $\tilde f' = 2f$. Then exactly the same argument in the previous paragraph goes through after multiplying \eqref{eq:rhobar-cong} for $(J,J')$ by $q+1$ (and $f$-periodically extending $u_i$, $u'_i$, $x_i$, $J$, $J'$ to $\Z/2f\Z$) to obtain a congruence modulo $p^{2f}-1$.
\end{proof}

\begin{remark} \label{rem:exceptional-case}
  In the exceptional case of Lemma~\ref{lem:carry-equation}$(iii)$, one can furthermore check that there are no $J,J'$ such that $\rhobar(\sigma,J) = \rhobar(\sigma',J')$ other than $\{J,J'\} = \{\varnothing,\Z/f\Z\}$.   Indeed, in niveau \(1\),  
the equality of $\sum_{i \in J} p^{-i} u_i$ and $\sum_{i \in J'} p^{-i} u'_i$ modulo $q-1$ says that two subsets of \(1,3,\ldots,3^{f-1}\) have sums that differ by
\((q-1)/2\) modulo \(q-1\). Since each such subset sum lies between \(0\) and
\((q-1)/2\), this forces the two subsets to be \(\varnothing\) and
\(\Z/f\Z\). In niveau \(2\) the argument is essentially the same, except that since $\varnothing,\Z/2f\Z$ are not profiles in niveau $2$, the conclusion is simply that there are no solutions $J,J'$.

  In particular $\Zss(\sigma) \not\subseteq \Zss(\sigma')$ and vice-versa. For this reason, in proving Theorem~\ref{thm:Zss} we will generally never be in the exceptional case of part $(iii)$ of the lemma, and will always have $x_i \in [-2,2]$. In arguments where the exceptional case is trivially excluded, e.g.\ because of a hypothesis that $\rhobar(\sigma,J) = \rhobar(\sigma',J')$ for some $J \not\in \{\varnothing,\Z/f\Z\}$, or that $u_i \neq p$ for some~$i$, we will pass over this without further mention.
\end{remark}

\begin{remark}\label{rem:all-zero}
    If $x_i = 0$ for all $i$, then Lemma~\ref{lem:carry-equation} immediately implies $u_i = u'_i$ for all $i$, and that $J,J'$ agree on all indices $i$ such that $u_i \neq 0$. Then \eqref{eq:rhobar-cong} implies that $t = t'$, and therefore $\sigma = \sigma'$. 
\end{remark}

As a first application of the lemma, we analyze the situation where $\rhobar(\sigma,J)$ is scalar for a niveau $2$ profile $J$.

\begin{cor}\label{cor:scalar-carry}
Suppose that $J$ is a profile such that $\rhobar(\sigma,J)$ is scalar. Then there exist unique integers $(y_i)$ such that
 \[ \eps_i^J u_i  = p y_{i-1} - y_i, \]
and if $J$ has niveau $2$ then $y_{i+f} = -y_i$ for all $i \in \Z/2f\Z$.
Furthermore  $|y_i| \le 1$ for all $i$.
\end{cor}

\begin{proof}
 Since $\rhobar(\sigma,J)$ is scalar, it is possible to choose $t' \in \Z/(q-1)\Z$ such that $\rhobar(\sigma,J) = \rhobar( (\ur;t') ,J')$ for some (indeed any) profile~$J'$ of the same niveau as $J$. Applying Lemma~\ref{lem:carry-equation} to the pair $(\ur;t')$ and $\sigma$ (in that order), we obtain unique integers $y_i$ such that
 \[ \eps_i^J u_i = py_{i-1} - y_i\] for all $i$, and $y_{i+f} = -y_i$ if $J$ has niveau $2$, giving the first part of the corollary. The bound $|y_i| \le 1$ follows by the same argument as for the bound in Lemma~\ref{lem:carry-equation}$(iii)$.
 \end{proof}

We saw in Remark~\ref{rem:same-f} that if $J$ has niveau $2$ and $\rhobar(\sigma,J)$ is scalar, then there also exists a profile $J'$ of niveau $1$ such that $\rhobar(\sigma,J') = \rhobar(\sigma,J)$. We now analyze the same situation with the niveaux reversed.

\begin{prop}
    \label{prop:niveau-2}
    Suppose that $J$ is a profile of niveau $1$ such that $\rhobar(\sigma,J)$ is scalar. Then there  also exists a profile $J'$ of niveau $2$ such that $\rhobar(\sigma,J') = \rhobar(\sigma,J)$ unless $\sigma = (\crit;t)$ for some $t$, in which case no such $J'$ exists.
\end{prop}

\begin{proof} Write $\sigma = (u_i;t)$.
The nonexistence of $J'$ when $\sigma = (\crit;t)$ is a consequence of Corollary~\ref{cor:scalar-carry}. Suppose $J'$ of niveau $2$ were to exist, and apply the Corollary. Since $u_i = p-1$ for all $i$ we have $y_{i-1} = y_i \in \{-1,1\}$ for all $i$, and recursively the sequence $y_i$ is constant. This contradicts $y_{i+f} = -y_i$. 

Now suppose that  $\sigma$ is non-critical and $\rhobar(\sigma,J)$ is scalar, with $J$ of niveau $1$. Corollary~\ref{cor:scalar-carry} gives integers $y_i \in [-1,1]$  such that $\eps_i^J u_i = py_{i-1} - y_i$ for all $i \in \Z/f\Z$. Since $\sigma$ is non-critical, some $y_i$ must be $0$. Without loss of generality suppose that $y_{-1} = 0$. Set $y'_i = y_i$ and $\eps_i^{J'} = \eps_i^J$ for $i \in \Z/2f\Z$  whose least non-negative residue lies in  $[0,f-1]$, and $y'_i = -y_i$, $\eps_i^{J'} = -\eps^J_i$ otherwise. Let $J'$ be the niveau $2$ profile implicitly defined by the choice of $\eps_i^{J'}$.  Now evidently
\begin{equation}\label{eq:niveau2-scalar}
     \eps_i^{J'} u_i = p y'_{i-1} - y'_i 
\end{equation}
for $i \neq 0,f$ (so that $\{i-1,i\}$ is contained in either $[0,f-1]$ or $[f,2f-1]$). But since $y'_{-1} = y'_{f-1} = 0$, the equation \eqref{eq:niveau2-scalar} holds for $i = 0,f$ as well, i.e., it holds for all $i \in \Z/2f\Z$.

Multiplying \eqref{eq:niveau2-scalar} by $p^{-i}$ and summing over $\Z/2f\Z$ proves that $\sum_{i\in J'} u_i p^{-i} - \sum_{i \not\in J'} u_i p^{-i} \equiv 0 \pmod{q^2-1}$, or in  other words that $\rhobar(\sigma,J')$ is scalar. However, we need the stronger statement that $\rhobar(\sigma,J') = \rhobar(\sigma,J)$, or equivalently that
\[ \sum_{i \in J'} u_i p^{-i} - (q+1) \sum_{i \in J} u_i p^{-i} \equiv 0 \pmod{q^2-1}.  \]
Since $\delta_J(i) = \frac{1}{2} (1 - \eps_i^J)$ and similarly for $J'$, the previous display rewrites as 
\[ \sum_{i= 0}^{2f-1} \tfrac{1}{2} (1-\eps_i^{J'}) u_i p^{-i} - \sum_{i=0}^{2f-1} \tfrac{1}{2} (1-\eps^J_i) u_i p^{-i} \equiv 0 \pmod{q^2-1}.  \]
From the definition of $\eps_i^{J'}$, the terms with $0 \le i < f$ in the two sums cancel, and the desired congruence is equivalent to
\[ \sum_{i=0}^{f-1} \eps_i^J u_i p^{-i} \equiv 0 \pmod{q^2-1}. \]
Substituting $\eps_i^J u_i = p y_{i-1} - y_i$, the terms on the left-hand side involving $y_i$ for $0 \le i < f-1$ telescope, leaving only the terms involving $y_{-1}$ and $y_{f-1}$. But $y_{-1} = y_{f-1} = 0$, so the previous congruence holds.
\end{proof}

As another early application of these ideas we show that the inclusion $\Zss((\ur;t))\subseteq\Zss((\crit;t))$ is simple. In fact we have the following stronger statement.

\begin{prop}\label{prop:critical-target}
Suppose that $
\Zss(\sigma)\subseteq\Zss((\crit;t)).$ Then $\sigma=(\ur;t)$ or $\sigma=(\crit;t).$
\end{prop}

\begin{proof} The proof will require the following claim. If \[ \sum_{i \in \Z/f\Z} a_i p^{f-i} \equiv \sum_{i \in \Z/f\Z} b_i p^{f-i} \pmod{p^{f-1} + \cdots + p+ 1} \]
with $a_i \in \{0,1,2\}$ and $b_i \in \{0,1\}$ for all $i$, then $a_i - b_i$ is constant. To see this, we treat the two sides as the integers $A = \sum_{i=1}^f a_i p^{f-i}$ and  $B = \sum_{i=1}^f b_i p^{f-i}$, so that the sums give the base $p$ expansions of $A$, $B$ respectively. Set $M = p^{f-1} + \cdots + p + 1$. Evidently $A-B \in \{-M,0,M,2M\}$. If $A-B = -M$ then $a_i = 0$ and $b_i = 1$ for all~$i$. If $A-B = 0$ then by uniqueness of base-$p$ expansions we have $a_i = b_i$ for all $i$. If $A-B = M$ then $A = \sum_{i=1}^{f} (b_i + 1)p^{f-i}$. Since $b_i + 1 < p$, this is the base-$p$ expansion of $A$, and $a_i = b_i + 1$ for all $i$. If $A - B = 2M$ then $a_i = 2$ and $b_i = 0$ for all $i$. In all cases the desired conclusion holds.

Without loss of generality suppose that $t = 0$. Write $\sigma = (u_i;t')$. Write $\rhobar(\sigma,\varnothing) = \rhobar((\crit;0),K)$ for a profile $K$, which we can take to have niveau $1$ by Remark~\ref{rem:same-f}. Similarly write $\rhobar(\sigma,\{i\}) = \rhobar((\crit;0),K_i)$ for a niveau $1$ profile $K_i$ for each $i$. Each $\rhobar((\crit;0),J)$ with $J$ of niveau $1$ has the form $\omega_0^{a} \oplus \omega_0^{b}$ with $a,b$ divisible by $p-1$. Applied to $\rhobar(\sigma,\varnothing)$ and $\rhobar(\sigma,\{i\})$ this shows that $t'$ and all of the $u_i$ are divisible by $p-1$. In particular $u_i \in \{0,p-1\}$ for all $i$. If~$K$ is $\Z/f\Z$, we replace it with~$\varnothing$ instead.  

 Now~\eqref{eq:rhobar-cong} shows that $t' \equiv (p-1) \sum_{i \in K} p^{f-i} \pmod{q-1}$, and another application of~\eqref{eq:rhobar-cong} for the profile $\{j\}$ shows, after dividing by $p-1$, that
\[p^{f-j} + \sum_{i \in K} p^{f-i} \equiv \sum_{i \in K_j} p^{f-i} \pmod{M} \]
for each $j$ in the set $S :=\{i:u_i=p-1\}$.
An application of our claim then shows that $j \not\in K$. Therefore $K\cap S = \varnothing$.

The equality $\det \rhobar(\sigma,\varnothing) = \det \rhobar((\crit;0),K)$ gives $2t' + (p-1) \sum_{i \in S} p^{f-i} \equiv 0 \pmod{q-1}$, or equivalently
\[ \sum_{i \in K} 2\cdot p^{f-i} + \sum_{i \in S} p^{f-i} \equiv 0 \pmod{M}. \]
Since $K\cap S=\varnothing$ and $K \neq \Z/f\Z$, another application of our claim gives $K = \varnothing$ and $S \in \{\varnothing,\Z/f\Z\}$. 
Thus $t' = 0$ and $\sigma=(\ur;0)$ or $\sigma=(\crit;0)$.
\end{proof}

\subsection{Further consequences of the carry equation}
\begin{defn}
Let $\sigma,\sigma'$ be a pair of phantom Serre weights.
\begin{enumerate}
    \item  We define $\cP(\sigma,\sigma')$ to be the set of profiles $J$ such that there exists a profile $K_J$ of the same niveau with $\rhobar(\sigma,J) = \rhobar(\sigma',K_J)$. 

    \item A \emph{pinning} is a choice, for each $J \in \cP(\sigma,\sigma')$, of a profile  $K_J$ of the same niveau with $\rhobar(\sigma,J) = \rhobar(\sigma',K_J)$. Given a pinning we write $x^J_i$ for the integers of Lemma~\ref{lem:carry-equation} applied to the pair of profiles $J, K_J$. 
\end{enumerate}
\end{defn}

We note the following subtlety in the definition of $\cP(\sigma,\sigma')$. If $J$ has niveau~$2$ and $\rhobar(\sigma,J)$ is scalar, then the existence of some $J'$ such that $\rhobar(\sigma,J) = \rhobar(\sigma',J')$ need not imply that $J \in \cP(\sigma,\sigma')$, since it may not be possible to choose $J'$ of niveau $2$. By Proposition~\ref{prop:niveau-2}, this happens precisely when $\sigma'$ is critical. To emphasize this, we highlight the case $\Zss(\sigma) \subseteq \Zss(\sigma')$.

\begin{cor}\label{cor:contain-profiles}
    Suppose that $\Zss(\sigma) \subseteq \Zss(\sigma')$. If $\sigma'$ is non-critical then the set $\cP(\sigma,\sigma')$ contains all profiles, while if $\sigma'$ is critical then $\cP(\sigma,\sigma')$ contains all profiles except for any niveau $2$ profiles $J$ with $\rhobar(\sigma,J)$ scalar.
\end{cor}

    Suppose for the remainder of the section that $\varnothing \in \cP(\sigma,\sigma')$, and we fix a pinning. It will be convenient to abbreviate $x_i := x^\varnothing_i$ and $K := K_\varnothing$, and we do so for the remainder of the paper. If $i \in \Z/2f\Z$ then $x_i$ means $x_{i \, \text{(mod $f$)}}$. Note that $\eps^\varnothing_i = 1$ for all $i$.  

    If also $J \in \cP(\sigma,\sigma')$, we shall need to understand the joint behavior of the integers $x_i$ and $x_i^J$, beyond what is immediately provided by Lemma~\ref{lem:carry-equation} (that they have the same parity). When jointly analyzing the equations coming from  Lemma~\ref{lem:carry-equation}$(i)$ for the pair of  profiles $\varnothing,J$ it is often better to use the sum and difference of those equations instead.   
    To that end, for any $J \in \cP(\sigma,\sigma')$ we write $s_i^J = \frac 12 (x_i^J + x_i)$  and $d_i^J = \frac 12 (x_i^J - x_i)$, so that $s_i^J,d_i^J \in [-2,2]$ apart from the exceptional case when $p=3$.   
    These are integers thanks to  Lemma~\ref{lem:carry-equation}$(ii)$. 
    Then
    \begin{equation}\label{eq:sum-cx}
        \frac{1}{2}(\eps_i^{K_J} + \eps_i^K)u'_i - \delta_{J^c}(i) u_i = p s^J_{i-1} - s^J_i
    \end{equation}
    and
        \begin{equation}\label{eq:diff-cx}
        \frac{1}{2}(\eps_i^{K_J} - \eps_i^K)u'_i + \delta_{J}(i) u_i = p d^J_{i-1} - d^J_i.
    \end{equation}
Since $\eps_i^{K_J} = \pm \eps_i^K$, we deduce the following.

\begin{lemma}\label{lem:alt}
Suppose that $\varnothing, J \in \cP(\sigma,\sigma')$. Then for each $i$ one of the following alternatives holds: either
\[ - \delta_{J^c}(i) u_i = p s^J_{i-1} - s^J_i \qquad \text{or} \qquad \delta_{J}(i) u_i = p d^J_{i-1} - d^J_i.  \]
\end{lemma}

    Here are three sample applications of the lemma that will be used in the next section. In what follows we will frequently use the following observations.
    One of the two alternatives of Lemma~\ref{lem:alt} will always be either $ps_{i-1}^J = s_i^J$ or $pd_{i-1}^J = d_i^J$, depending on whether $i \in J$ or $i \not\in J$. Outside the exceptional case, we have 
    $s_{i-1}^J , s_i^J \in [-2,2]$ but $p \ge 3$, so that $ps_{i-1}^J = s_i^J$ implies $s_{i-1}^J = s_i^J = 0$;\ and similarly for $p d_{i-1}^J = d^J_i$.   Also outside the exceptional case,  if  $x_i < 0$ then $s_i^J \le 0$ and $d_i^J \ge 0$:\ consider that if $x_i = -1$ then $x_i^J = \pm 1$ and if $x_i = -2$ then $x_i^J \in \{-2,0,2\}$.
    
    \begin{lemma}\label{lem:I-profile}
    Let $I$ be the niveau $1$ profile $\{ i : 0 < u_i < p \}$. If $\varnothing, I \in \cP(\sigma,\sigma')$ and $x_i < 0$, then $(u_i,u_{i+1}) \neq (1,p)$.
    \end{lemma}

    \begin{proof}
        Assume $u_i = 1$, so that $i \in I$.  We claim that $s^I_i = 0$. In the first alternative of Lemma~\ref{lem:alt}  we have $p s^I_{i-1} = s^I_i$, implying $s_{i-1}^I = s^I_i = 0$. In the second alternative we get $1 = p d_{i-1}^I - d_i^I$. But the hypothesis $x_i < 0$ implies that $d_i^I \in \{0,1,2\}$, and so the only possibility is that $p=3$, $d_{i-1}^I = 1$, and $d_i^I = 2$. The latter forces  $x_i = -2$ and $x_i^I = 2$, and so again $s^I_i = 0$. 

        Now suppose in addition that $u_{i+1} = p$, so that $i+1 \not\in I$,        
        and 
        consider Lemma~\ref{lem:alt} applied at $i+1$.  Since $s_i^I = 0$ by the previous paragraph, the first alternative becomes $-p = -s^I_{i+1}$, which contradicts the bounds on $s^I_{i+1}$. The second alternative gives $p d^I_{i} - d^I_{i+1} = 0$, and therefore $d^I_i = d^I_{i+1} = 0$. But $s^I_i = 0$ implies $d^I_i = |x_i| \ge 1$, and again we obtain a contradiction.
    \end{proof}

\begin{lemma}\label{lem:P-profile}
      Let $P$ be the niveau $2$ profile satisfying 
      $P \cap \{0,\ldots,f-1\} = \{ i : u_i = p \}$. If $\varnothing\in \cP(\sigma,\sigma')$,  $u_0 = 0$, $(u_{i-1},u_i) \neq (0,p)$ for all $i$, and $x_i < 0$ for all $i$, then $P \not\in \cP(\sigma,\sigma')$.
\end{lemma}

\begin{proof} We argue by contradiction, so let us suppose that $P \in \cP(\sigma,\sigma')$.
Suppose $1 \le i \le f-1$. We claim, first, that if $s^P_i = 0$ then also $s^P_{i-1} = 0$. In the second alternative of Lemma~\ref{lem:alt}, the left-hand side is always divisible by $p$ for the profile $P$, and therefore $d_i^P = 0$. But $d_i^P = s_i^P = 0$ implies $x_i = 0$, contradicting our hypothesis, so it must be the first alternative that holds. In that alternative the right-hand side is divisible by $p$. But if the left-hand side is divisible by $p$, then it must be $0$ (because if $u_i = p$ then $\delta_{P^c}(i) = 0$). Therefore also $s^P_{i-1} = 0$. 

In the converse direction, we claim that if $s_{i-1}^P = 0$ and $u_i \neq p$ then $s_i^P = u_i = 0$, again for $1 \le i \le f-1$. As in the previous paragraph, the second alternative gives $d_i^P = 0$. Since  $u_i \neq p$ by hypothesis, the left-hand side of this alternative is $0$, and we conclude $d_{i-1}^P = 0$. But $d_{i-1}^P = s_{i-1}^P = 0$ contradicts $x_{i-1} \neq 0$, so again it is the first alternative that holds. Now $s_i^P = \delta_{P^c}(i) u_i = u_i$. But $x_i < 0$ implies $s_i^P \le 0$, and this combined with $s_i^P = u_i \ge 0$ gives $s_i^P = u_i = 0$.

Finally consider $i = 0$. Since $u_0 = 0$ the alternatives give either $s_{-1}^P = s_0^P = 0$ or $d_{-1}^P = d_0^P = 0$. Also since $u_0 = 0$ the equation $\eps_0^K u'_0 = p x_{-1} - x_0$ gives $|x_{-1}| \le 1$. Therefore $x_{-1}  = -1$, and $x^P_{-1} = \pm 1$.

Suppose first that $d^P_{-1} = d^P_0 = 0$, so that $x_{-1} = x^P_{-1} = -1$. Then $x^P_{f-1} = -x^P_{-1} = 1$, and $s_{f-1}^P = 0$. Iteratively applying the conclusion in the first paragraph of the proof, we conclude that $s_0^P = 0$. But $s_0^P = d_0^P = 0$ implies $x_0 = 0$, a contradiction. So we must instead have $s^P_{-1} = s^P_0 = 0$. Iteratively applying the conclusion in the second paragraph of the proof, which is possible because of the hypothesis $(u_{i-1},u_i) \neq (0,p)$ for all $i$, we obtain $s^P_{f-1} = 0$. But $s^P_{-1} = s^P_{f-1} = 0$ is again a contradiction, because $x_{f-1}^P = -x_{-1}^P \neq x_{-1}^P$.
\end{proof}

\begin{lemma}\label{lem:x-negative-u-positive}
Suppose that $\varnothing, J \in \cP(\sigma,\sigma')$. 
\begin{enumerate}
    \item If $x_{i-1},x_i < 0$, $u_i > 0$, and $s_{i-1}^J = 0$ then $i \in J$.

    \item If $x_{i-1},x_i,x_{i+1} < 0$, $u_{i-1},u_i,u_{i+1} > 0$, and $J \cap \{i-1,i,i+1\} = \{i\}$, then $u_{i-1} = p + 1 + x_{i-1}$ and $u_i = p - d_i^J$. 
\end{enumerate}
\end{lemma}

\begin{proof}
$(i)$ Suppose $i \not\in J$. Then the alternatives of Lemma~\ref{lem:alt} are
\[ -u_i = ps_{i-1}^J - s_i^J \qquad \textrm{or} \qquad pd_{i-1}^J = d_i^J.\]
The hypotheses $s_{i-1}^J = 0$, $x_{i-1} < 0$ rule out $d_{i-1}^J = 0$, so the second cannot hold. The first alternative then gives $u_i = s_i^J$. But $x_i < 0$ implies $s_i^J \le 0$, contradicting the hypothesis $u_i > 0$.

$(ii)$  Since $u_i > 0$ and $d_i^J \ge 0$, Lemma~\ref{lem:alt} applied at $i \in J$ gives either 
\begin{equation}\label{eq:u-pos-alt-1}
 s_{i-1}^J = s_i^J = 0 \qquad \textrm{or} \qquad d_{i-1}^J = 1,\ u_i = p - d_i^J.
 \end{equation}
Lemma~\ref{lem:alt} applied at $i+1 \not\in J$ gives either 
\[ -u_{i+1} = ps_i^J - s_{i+1}^J \qquad \textrm{or} \qquad d_i^J = d_{i+1}^J = 0. \]
Either of these alternatives contradicts the first alternative of \eqref{eq:u-pos-alt-1}. Indeed if $s_i^J = 0$ then we cannot have $u_{i+1} = s_{i+1}^J$ (because the left-hand side is positive and the right-hand side is not), nor can we have $d_i^J = 0$ (because $x_i \neq 0$). Therefore it is the second alternative in \eqref{eq:u-pos-alt-1} that holds. In particular $u_i = p - d_i^J$.

Now apply Lemma~\ref{lem:alt} at $i-1\not\in J$. Since $d_{i-1}^J = 1$, the second alternative cannot hold, and the first alternative gives
\[ u_{i-1} = -ps_{i-2}^J + s_{i-1}^J.\]
Since $u_{i-1} > 0$ and $s_{i-1}^J \le 0$, we must have $s_{i-2}^J = -1$ and $u_{i-1} = p + s_{i-1}^J$. 
Since $s_{i-1}^J = x_{i-1} + d_{i-1}^J = x_{i-1} + 1$, the lemma follows.
\end{proof}

\section{The six alternatives}

Suppose that $\Zss(\sigma) \subseteq \Zss(\sigma')$. If $\sigma'$ is non-critical then the set $\cP(\sigma,\sigma')$ contains all profiles, while if $\sigma'$ is critical then $\cP(\sigma,\sigma')$ contains all profiles except for any niveau $2$ profiles $J$ with $\rhobar(\sigma,J)$ scalar.

Assume  that $\sigma \neq \sigma'$ and fix a pinning. By Remark~\ref{rem:all-zero} the following four alternatives are  exhaustive and mutually exclusive.
\begin{enumerate}
    \item[($+$)] We have $x_i > 0$ for all $i$.
    \item[($P$)] There exists $j$ such that $x_{j-1} \le 0$ and $x_j > 0$.
    \item[($N$)] We have $x_i \le 0$ for all $i$, and there exists $j$ such that $x_{j-1} = 0$, $x_j < 0$.
    \item[($-$)] We have $x_i < 0$ for all $i$.
\end{enumerate}
We refine alternatives ($P$) and ($N$) into two subcases each, for a total of six alternatives, no longer mutually exclusive since both subcases may occur for different values of $j$.
\begin{enumerate}
    \item[($P.a$)] There exists $j$ such that $x_{j-1} \le 0$ and $x_j > 0$, and $u_j > 0$.
    \item[($P.b$)] There exists $j$ such that $x_{j-1} \le 0$ and $x_j > 0$, and $u_j = 0$.
    \item[($N.a$)] We have $x_i \le 0$ for all $i$, and there exist integers  $j' \ge j$ such that $x_{j-1} = 0$, $x_j,\ldots,x_{j'} < 0$, $u_j = \cdots = u_{j'} = 0$, and $u_{j'+1} > 0$.
    \item[($N.b$)] We have $x_i \le 0$ for all $i$, and there exists $j$ such that $x_{j-1} = 0$, $x_j < 0$, but there is no $j'$ as in ($N.a$) for this $j$.
\end{enumerate}

We introduce the following condition on pinnings.

\begin{defn}
    Suppose $\Zss(\sigma) \subseteq \Zss(\sigma')$. We say that a pinning of the pair $\sigma,\sigma'$ is \emph{minimal} if $\sum_{i \in \Z/f\Z} |x_i|$ is as small as possible as one varies over all $\widetilde{\sigma} \sim \sigma$ and all pinnings of $\widetilde{\sigma},\sigma'$. 

    Note that, according to this definition, a pair $\sigma,\sigma'$ with $\Zss(\sigma) \subseteq \Zss(\sigma')$ need not have a minimal pinning, but a minimal pinning always exists after possibly replacing $\sigma$ with another phantom Serre weight $\widetilde{\sigma}$ with $\widetilde{\sigma} \sim \sigma$. In what follows, when we take a minimal pinning to be given, we mean literally on $\sigma$ rather than after replacement.
\end{defn}

\begin{lemma}\label{lem:minimal-reductions}
Let $\sigma,\sigma'$ be phantom Serre weights such that $\varnothing \in \cP(\sigma,\sigma')$. There is no index $j$ with $x_j = 2$, with $x_{j} = 1$ and $x_{j+1} < 0$, or with $x_j = -2$ and $x_{j+1} > 0$.

If furthermore $\Zss(\sigma) \subseteq \Zss(\sigma')$ is equipped with a minimal pinning, then:
\begin{enumerate}
    \item There is no index $j$ with $(u_{j},u_{j+1}) = (1,0)$ and $x_{j} > 0$.
    \item There is no index $j$ with $(u_j,u_{j+1}) = (0,p)$ and $x_j < 0$.
\end{enumerate}
\end{lemma}

\begin{proof}
Lemma~\ref{lem:carry-equation} applied at $J = \varnothing$ gives $u'_{j+1} = |u_{j+1} + px_{j} - x_{j+1}| $. Taking $x_j = 2$ or $x_j = 1$, $x_{j+1} < 0$ or $x_j = -2$, $x_{j+1} > 0$ gives a contradiction to $u'_{j+1} \in [0,p]$. 

For $(i)$, Lemma~\ref{lem:carry-equation} gives
\begin{align*}
  \eps^K_{j} u'_{j} - 1 &= p x_{j-1} - x_{j}\\
\eps^K_{j+1}   u'_{j+1}      - 0   & = p x_{j} - x_{j+1}. 
\end{align*}
These equations rearrange to
\begin{align*}
  \eps^K_{j} u'_{j} - 0 &= p x_{j-1} - (x_{j}-1)\\
\eps^K_{j+1}   u'_{j+1}      - p  & = p (x_{j}-1) - x_{j+1} 
\end{align*}
which, along with $\eps_i^K u'_i - u_i = p x_{i-1} - x_i$ for $i \neq j,j+1$ must provide the unique output of Lemma~\ref{lem:carry-equation} for the pair $(\mu_j(\sigma), \sigma')$ with $(J,K_J) = (\varnothing,K)$. (Note that $\rhobar(\mu_j(\sigma), \varnothing) = \rhobar(\sigma, \varnothing) = \rhobar(\sigma',K)$, so the lemma applies.) Since $x_{j} > 0$, the sum $|x_{j} - 1| + \sum_{i \neq j} |x_i|$ is strictly smaller than $\sum_i |x_i|$, contradicting minimality of the original pinning. This establishes $(i)$, and the proof of $(ii)$ is similar.
\end{proof}

In preparation for the next section, we establish the following reductions.

\begin{thm}\label{thm:preliminaries}
Suppose $\Zss(\sigma) \subseteq \Zss(\sigma')$, equipped with a minimal pinning. Then:
\begin{itemize}
    \item In case {\upshape($+$)} we have $\sigma,\sigma'  = (\ur;t),(\crit;t)$ or $(\BT;t),(\St;t)$ for some $t$.
    \item In case {\upshape($P$)}, and for $j$ as in that case, treating $j$ as an integer there exists a least $j' \ge j$ such that $x_j,\dots,x_{j'} > 0$ and $x_{j'+1} \le 0$. Then
    $u_{j+1} = \cdots = u_{j'+1} = 0$.
    \item In case {\upshape($N.b$)}, and for $j$ as in that case, we have $u_{j+1} = 0$.
    \item In case {\upshape($-$)} we have $\sigma,\sigma'  = (\ur;t),(\crit;t)$ or $(\St;t),(\BT;t)$ for some $t$.
\end{itemize}
\end{thm}

\begin{proof} Throughout the proof we take $\sigma = (u_i;t)$ and $\sigma' = (u'_i;t')$.

\medskip
\textbf{The case $(+)$.} By hypothesis $x_i > 0$ for all $i$. Lemma~\ref{lem:minimal-reductions} tells us that $x_i = 1$ for all $i$. Lemma~\ref{lem:carry-equation} gives $\eps_i^K u'_i = u_i + (p-1)$. Since $u'_i \in [0,p]$ we have $u_i \in \{0,1\}$ for all $i$ and $K = \varnothing$. By Lemma~\ref{lem:minimal-reductions}$(i)$ there is no $j$ with $(u_{j-1},u_j) = (1,0)$. Therefore either $u_i = 0$ for all $i$, or $u_i = 1$ for all $i$. In the first case $u'_i = p-1$ for all $i$, and~\eqref{eq:rhobar-cong} with $J = J' = \varnothing$ shows that $t' = t$, giving the first option $\sigma,\sigma' = (\ur;t),(\crit;t)$. If instead $u_i = 1$ for all $i$ then in the same manner we obtain $\sigma,\sigma' = (\BT;t),(\St;t)$.

\medskip
\textbf{The case $(P)$}. By Lemma~\ref{lem:minimal-reductions} we in fact have $x_j = \cdots = x_{j'} = 1$ and $x_{j'+1} = 0$. The equation $\eps_{j'+1}^K u'_{j'+1} = u_{j'+1} + px_{j'} - x_{j'+1} = u_{j'+1} + p$ forces $u_{j'+1} = 0$. Similarly for $j < i \le j'$ the equation $\eps_{i}^K u'_{i} = u_{i} + px_{i-1} - x_{i} = u_{i} + p - 1$ gives $u_i \in \{0,1\}$. But $(u_i,u_{i+1}) = (1,0)$ is ruled out by Lemma~\ref{lem:minimal-reductions}$(i)$, and since already $u_{j'+1} = 0$, we iteratively conclude that $u_i = 0$ for $j < i \le j'$.

\medskip
\textbf{The case $(N.b)$}. By hypothesis $x_i \le 0$ for all $i$, $x_{j-1} = 0$, and $x_j < 0$. If $u_j = 0$, then the hypothesis that we are not in case $(N.a)$ already gives $u_{j+1} = 0$, so we may assume $u_j > 0$. Lemma~\ref{lem:carry-equation} for $J = \varnothing$ gives $u'_j = u_j - x_j$. In particular $u_j \le p + x_j$ and $j \not\in K$.


Lemma~\ref{lem:alt} applied to the profile $\{j\}$ gives either
\begin{equation}\label{eq:Nb}
    ps_{j-1}^{\{j\}} = s_j^{\{j\}} \qquad \text{or} \qquad u_j = pd_{j-1}^{\{j\}} - d_j^{\{j\}}.
\end{equation}
In the first alternative $ s_{j-1}^{\{j\}} = s_j^{\{j\}} = 0$, so that $x_{j}^{\{j\}}  = d_j^{\{j\}} = - x_j$. In the second alternative since $x_j < 0$ we have $d_j^{\{j\}}\ge 0$, and $u_j > 0$ implies $d_{j-1}^{\{j\}} = 1$, $x_{j-1}^{\{j\}} = 2$, and $u_j = p - d_j^{\{j\}}$. Comparing the latter with $u_j \le p + x_j$ we have $-x_j \le  d_j^{\{j\}}$. With $x_j < 0$ this is only possible if $d_j^{\{j\}} =  - x_j$. Thus in either alternative we obtain $ d_j^{\{j\}} = - x_j$ and $s_j^{\{j\}} = 0$.

Lemma~\ref{lem:alt} applied to the profile $\{j\}$ at the index $j+1$ then  gives either
\[ u_{j+1} = s_{j+1}^{\{j\}} \qquad \text{or}\qquad  d_{j+1}^{\{j\}} = -px_j. \]
The latter is impossible because $x_j \neq 0$ and $|d_{j+1}^{\{j\}}|\le 2$, so the first alternative must hold.

Assume for the sake of contradiction that $u_{j+1} > 0$. Then $s_{j+1}^{\{j\}} > 0$. Combined with $x_{j+1} \le 0$ the only possibility is $x_{j+1} = 0$, $u_{j+1} = s_{j+1}^{\{j\}} = 1$. Then also $d_{j+1}^{\{j\}} = 1$. Now consider Lemma~\ref{lem:alt} applied to the profile $\{j\}$ at the index $j+2$. If $f \neq 2$, then $j+2 \not\in\{j\}$, and the lemma gives either
\[ -u_{j+2} = ps_{j+1}^{\{j\}} - s_{j+2}^{\{j\}} \qquad \text{or} \qquad pd_{j+1}^{\{j\}} = d_{j+2}^{\{j\}}.\]
But $s_{j+1}^{\{j\}} = d_{j+1}^{\{j\}} = 1$ gives a contradiction to both alternatives.

Finally, suppose $f = 2$.  Returning to the profile $\{j\}$, since $s_{j+1}^{\{j\}} = s_{j-1}^{\{j\}} = 1$ the first alternative of \eqref{eq:Nb} is ruled out. The second alternative with $d_{j+1}^{\{j\}} = d_{j-1}^{\{j\}} = 1$ gives $u_j = p + x_j$ and $u'_j = p$. Since $\eps_{j+1}^K u'_{j+1} = u_{j+1} + p x_j - x_{j+1} = 1 + px_j$, we can rule out $x_j = -2$. Therefore $x_j = -1$, $u_j = p-1$, $u'_{j+1} = p-1$, and $K = \{j+1\}$. To fix ideas suppose that $j = 0$. Then~\eqref{eq:rhobar-cong} gives $t' - t \equiv p-1 \pmod{p^2-1}$. 
 We conclude that $\sigma = (p-1,1;t)$ and $\sigma' = (p,p-1;t+(p-1))$. However, one can check explicitly for the niveau $2$ profile $J = \{0,3\} \subset \Z/4\Z$ that $\rhobar(\sigma,J) = \omega_0^t \otimes ((\omega'_0)^{2p-1} \oplus (\omega'_0)^{2p^3-p^2}) \not\in \Zss(\sigma')$, a contradiction. We conclude that $u_{j+1} = 0$.

\medskip
\textbf{The case $(-)$}.
  Suppose first that $u_i > 0$ for all $i$. Consider the singleton profiles $\{i\}$.
 Lemma~\ref{lem:x-negative-u-positive}$(ii)$ applies to each of these profiles. Therefore $u_i = p + 1 + x_i$  and $u_i = p - d_i^{\{i\}}$ for all $i$. Comparing, we find that $d_i^{\{i\}} = -x_i - 1$ for all $i$. 

 One possibility is that $x_i = -1$ for all $i$. Then $u_i = p$ for all $i$, and \[ \eps_i^K u'_i = u_i + px_{i-1} - x_i = u_i - (p-1)\] for all $i$ gives $u'_i = 1$ for all $i$ and $K = \varnothing$. Now~\eqref{eq:rhobar-cong} gives $t' = t$ and we obtain $\sigma = (\St;t)$ and $\sigma' = (\BT;t)$. This is the second option in the theorem.
 
 Another possibility is that $x_i = -2$ for all $i$. Then $u_i = p-1$ for all $i$, and
 \[ \eps_i^K u'_i = u_i + px_{i-1} - x_i = u_i - 2(p-1).\] 
Therefore $u'_i = (p-1)$ and $K = \Z/f\Z$.  Then~\eqref{eq:rhobar-cong} gives $t' = t$ and we obtain $\sigma = \sigma' = (\crit;t)$. This contradicts the hypothesis that $\sigma \neq \sigma'$.
 
Now suppose neither $x_i = -1$ for all $i$ nor $x_i = -2$ for all $i$. Then there exists some~$j$ with $x_{j-1} = -2$ and $x_{j} = -1$. From $x_{j-1} = -2$ we get $d_{j-1}^{\{j-1\}} = 1$, while $x_j = -1$ gives $u_j = p$. Lemma~\ref{lem:alt} applied at $j \not\in \{ j-1 \}$ gives
    \[ -u_j = p s_{j-1}^{ \{j-1\} } - s_j^{ \{j-1\} } \qquad \textrm{or} \qquad pd_{j-1}^{ \{j-1\} } = d_j^{ \{j-1\}}. \]
    The second alternative is ruled out because $d_{j-1}^{\{j-1\}} = 1$. Since $u_j = p$, the first alternative yields $s_j^{\{j-1\}} = 0$. Now Lemma~\ref{lem:x-negative-u-positive}$(i)$ implies $j+1 \in \{j-1\}$. We conclude in this case that $f=2$. Supposing without loss of generality that $x_{0} = -1$ we have $\sigma = (p,p-1;t)$. Lemma~\ref{lem:carry-equation} then gives $(u'_0,u'_1) = (p-1,1)$ and $K = \{0\}$, whence $t' = t - (p-1)$ by~\eqref{eq:rhobar-cong}. Therefore     
    $\sigma' = (p-1,1;t - (p-1))$. This is exactly the reverse of the $f=2$ subcase of $(N.b)$. One can check explicitly for the niveau $2$ profile $J = \{0,1\} \subset \Z/4\Z$ that $\rhobar(\sigma,J) = \omega_0^{t-(p-1)} \otimes ((\omega'_0)^{2p-p^2} \oplus (\omega'_0)^{2p^3-1}) \not\in \Zss(\sigma')$, another contradiction. This completes the case where $u_i > 0$ for all $i$.
    
We may therefore suppose that $u_i = 0$ for some $i$. We aim to prove that $u_i = 0$ for all $i$. Without loss of generality suppose that $u_0 = 0$. Lemma~\ref{lem:minimal-reductions} gives $(u_{i-1},u_i) \neq (0,p)$ for all $i$, and then Lemma~\ref{lem:P-profile} gives $P \not\in \cP(\sigma,\sigma')$.  Since $\Zss(\sigma) \subseteq \Zss(\sigma')$, the conclusion $P \not\in \cP(\sigma,\sigma')$ implies that $\rhobar(\sigma,P)$ is scalar. Now Corollary~\ref{cor:scalar-carry} gives integers $y_i \in [-1,1]$ such that
\[ \eps_i^P u_i = p y_{i-1} - y_i \]
for $0 < i < f$. Since $y_{f-1} = -y_{-1}$ we also have $\eps_0^P u_0 = -py_{f-1} - y_0$. The assumption $u_0 = 0$ forces $y_0 = y_{f-1} = 0$. For $0 < i < f$, by definition $\eps_i^P > 0$ if and only if $u_i \neq p$, and so
we have the following possibilities:
\[
\renewcommand{\arraystretch}{1.25}
\setlength{\tabcolsep}{14pt}
\begin{array}{c!{\vrule}rr}
\toprule
u_i & \multicolumn{1}{c}{y_{i-1}} & \multicolumn{1}{c}{y_i} \\
\midrule
0   & 0  & 0 \\
1   &  0 & -1 \\
p-1 &  1 &  1 \\
p   & -1 &  0 \\
\bottomrule
\end{array}
\]
If some $y_i$ is nonzero, choose the first such $i>0$.
Then $y_{i-1}=0$, so the table gives $u_i=1$ and $y_i=-1$.
Applying the table at $i+1$ gives $u_{i+1}=p$. But $(u_i,u_{i+1})=(1,p)$ contradicts
 Lemma~\ref{lem:I-profile};\ therefore $y_i = 0$ for all $i$, and $u_i = 0$ for all $i$.

Finally, from $u_i = 0$ for all $i$ we have $\eps_i^{K} u'_i = px_{i-1} - x_i$ for all $i$. If some $x_{i-1}$ were $-2$, we would obtain a contradiction to $u'_i \in [0,p]$. Therefore $x_i  = -1$ and $u'_i = p-1$ for all $i$, with $K = \Z/f\Z$. It follows that $t' = t$ and $\sigma,\sigma' = (\ur;t), (\crit;t)$, as claimed.
\end{proof}

\section{Four inclusions}

In this section we prove the following key result.

\begin{thm}\label{thm:key}
Suppose $\Zss(\sigma) \subseteq \Zss(\sigma')$, equipped with a minimal pinning. Assume $\sigma \neq \sigma'$, and that $\{ \sigma,\sigma'\} \neq \{ (\BT;t),(\St;t)\}$ for all $t$. Then
\begin{itemize}
    \item In case {\upshape($+$)} we have $\sigma,\sigma'  = (\ur;t),(\crit;t)$  for some $t$.
    \item In case {\upshape($P.a$)}, and for $j$ as in that case, we have $\Zss(\mu_{j}(\sigma)) \subseteq \Zss(\sigma')$.
        \item In case {\upshape($P.b$)}, and for $j$ as in that case,  we have $\Zss(\theta_{j}(\sigma)) \subseteq \Zss(\sigma')$.
    \item In case {\upshape($N.a$)}, and for $j'$ as in that case,  we have   $\Zss(\nu_{j'}(\sigma)) \subseteq \Zss(\sigma')$.
    \item In case {\upshape($N.b$)}, and for $j$ as in that case,  we have  $\Zss(\theta_{j}(\sigma)) \subseteq \Zss(\sigma')$.
    \item In case {\upshape($-$)} we have $\sigma,\sigma'  = (\ur;t),(\crit;t)$ for some $t$.
\end{itemize}
Moreover at least one of these cases holds.
\end{thm}

\begin{remark}
    If $\Zss(\sigma) \subseteq \Zss(\sigma')$ is equipped with a minimal pinning, then by Remark~\ref{rem:all-zero} together with the minimality hypothesis, we have $\sigma = \sigma'$ if and only if $x_i = 0$ for all $i$. Thus the hypothesis $\sigma \neq \sigma'$ serves only to ensure that at least one of the cases in the theorem holds.
\end{remark}

\begin{remark}\label{rem:key-critical}
    If $\Zss(\sigma) \subseteq \Zss(\sigma')$ with $\sigma'$ critical, and $\sigma \neq \sigma'$, then Proposition~\ref{prop:critical-target} implies that $\sigma$ is scalar. But then either $x_i = 1$ for all $i$, or $x_i = -1$ for all $i$. In particular, if $\sigma'$ is critical in Theorem~\ref{thm:key} then it is either $(+)$ or $(-)$ that holds, and not one of the four middle cases.

    The same argument shows that if $\Zss(\sigma) \subseteq \Zss(\sigma')$ and furthermore $u_j = x_{j'} = 0$ for some $j,j'$, then $\sigma'$ is not critical:\ if it were, $u_j = 0$ implies that $\sigma$ is scalar, and $x_{j'} = 0$ provides a contradiction.
\end{remark}

\begin{remark}\label{rem:crucial}
Crucially, each of the operations in the four middle cases of Theorem~\ref{thm:key} is valid and non-invertible. To begin with, in the applications of $\mu_{j}$ and $\theta_{j}$ above we have $u_{j+1} = 0$ by Theorem~\ref{thm:preliminaries}, while $u_{j'} = 0$ in the application of $\nu_{j'}$ by definition. Then:
\begin{itemize}
    \item In case $(P.a)$ we are given $u_j > 0$, and since $x_j > 0$, $u_{j+1} = 0$ we have $u_j \neq 1$ by Lemma~\ref{lem:minimal-reductions}$(i)$. Therefore $\mu_{j}$ is valid and non-invertible.

    \item In case $(P.b)$, the operation $\theta_{j}$ is valid because $u_j \neq p$ (indeed we are given $u_j = 0$), and $\theta_{j}$ is always non-invertible.

    \item In case $(N.a)$, we are given $x_{j'} < 0$, $u_{j'} = 0$. Therefore $u_{j'+1} \neq p$ by Lemma~\ref{lem:minimal-reductions}$(ii)$, and $\nu_{j'}$ is valid and non-invertible.

    \item In case $(N.b)$, if  $u_j = p$ then we would have $u'_j = p + px_{j-1} - x_j > p$. Therefore $u_j\neq p$, and so $\theta_{j}$ is valid and non-invertible.
\end{itemize}
\end{remark}

The cases $(+)$ and $(-)$ of Theorem~\ref{thm:key} have already been proved in Theorem~\ref{thm:preliminaries}. To organize the rest of the proof, each of the other four cases will be established as a separate proposition, stated somewhat more generally than above. Before this, however, we record some features common to all four cases. 

\begin{defn}
    If $\sigma$ is a phantom Serre weight and $J$ is a profile of niveau 1, we set
    \[ \cL(\sigma,J) = t + \sum_{i \in J} p^{-i} u_i, \quad \cR(\sigma,J) = t + \sum_{i \in J^c} p^{-i} u_i \quad \text{in}\ \Z/(q-1)\Z.\]
Similarly if $J$ has niveau 2, we set
    \[ \cL(\sigma,J) = (q+1)t + \sum_{i \in J} p^{-i} u_i, \quad \cR(\sigma,J) = (q+1)t + \sum_{i \in J^c} p^{-i} u_i \quad \text{in}\ \Z/(q^2-1)\Z.\] Thus
    \[ \rhobar(\sigma,J) = \begin{cases}
\omega_0^{\cL(\sigma,J)} \oplus \omega_0^{\cR(\sigma,J)} & J  \text{ of niveau 1} \\
(\omega'_0)^{\cL(\sigma,J)} \oplus (\omega'_0)^{\cR(\sigma,J)}  & J \text{ of niveau 2}.       
    \end{cases}\]
Also set $d(\sigma) := 2t + \sum_{i \in \Z/f\Z} p^{-i} u_i$, so that  
$\det \rhobar(\sigma,J) = \omega_0^{d(\sigma)}$ for all profiles $J$.  If $\Zss(\sigma) \subseteq \Zss(\sigma')$, then evidently we must have $d(\sigma) = d(\sigma')$.
\end{defn}

\begin{defn}
  We write
  \[ \cL(\sigma) = \{ \cL(\sigma,J) : \text{all profiles}\ J \} \subset \Z/(q-1)\Z \ {\textstyle\coprod}\  \Z/(q^2-1)\Z \] and similarly for $\cR(\sigma)$. 
\end{defn}

\begin{remark}\label{rem:inclusion-critical-rem}
  Since $\cR(\sigma,J) = d(\sigma) - \cL(\sigma,J)$ for all niveau $1$ profiles (with an analogous formula for niveau $2$ profiles) it follows that if $\sigma'$ is non-critical then 
$\Zss(\sigma) \subseteq \Zss(\sigma')$ if and only if $d(\sigma) = d(\sigma')$ and $\cL(\sigma) \subseteq \cL(\sigma')$. If $\sigma'$ is critical, the `if' direction still holds;\ the problem with the `only-if' direction is that  $\rhobar(\sigma,J)$ might be scalar for some niveau $2$ profiles $J$, and the associated values of $\cL(\sigma)$ will not be in the $\Z/(q^2-1)\Z$ part of $\cL(\sigma')$.
\end{remark}

One checks easily from the definitions that $d(f_i(\sigma)) = d(\sigma)$ for any valid  operation~$f_i$. If we are given $\Zss(\sigma) \subseteq \Zss(\sigma')$, then the value $d := d(\sigma)$ is common to all three of $\sigma, \sigma', f_i(\sigma)$. So, if we want to establish that $\Zss(f_i(\sigma)) \subseteq \Zss(\sigma')$, it suffices to prove that $\cL(f_i(\sigma)) \subseteq \cL(\sigma')$.

Since $J$ is a profile if and only if $J^c$ is a profile and $\cR(\sigma,J) = \cL(\sigma,J^c)$, we see immediately that $\cL(\sigma) = \cR(\sigma)$ for any $\sigma$.  Thus $\cL(f_i (\sigma),J) \in \cL(\sigma')$ if and only if $\cR(f_i (\sigma),J) \in \cR(\sigma')$, if and only if $\cL(f_i (\sigma),J^c) \in \cL(\sigma')$. Thus to prove the inclusion $\cL(f_i(\sigma)) \subseteq \cL(\sigma')$, it is enough to establish for all $J$ that either $\cL(f_i(\sigma),J)$ or $\cL(f_i(\sigma),J^c)$ is in $\cL(\sigma')$;\ that is, one only needs to consider one profile out of each pair $\{J,J^c\}$. 

In fact, for each $f_i$ there is a family of profiles $J$ for which one sees in one line from the definitions that  $\cL(f_i(\sigma), J) = \cL(\sigma, J')$ for suitable $J'$.  If $\sigma'$ is non-critical then $\cL(\sigma) \subseteq \cL(\sigma')$, and so we conclude $\cL(f_i(\sigma),J) \in \cL(\sigma')$ for such $J$.
\[
\begin{array}{c c c}
\toprule
f_i & \text{condition on } J & J' \\
\midrule
\mu_i    & i \in J,\; i+1 \in J              & J \\
\theta_i & i \in J,\; i+1 \notin J     & J \\
\nu_i    & i \in J,\; i+1 \notin J     & J \triangle \{i+1\} \\
\bottomrule
\end{array}
\]
Here and throughout we will use the following convention. If $J$ has niveau $1$ and $S \subset \Z/f\Z$, then the symmetric difference $J \triangle S$ has its usual meaning. However, if $J$ has niveau $2$, then $S$ should be viewed as a subset of $\Z/f\Z$ extended $f$-periodically to a subset of $\Z/2f\Z$, so that $J \triangle S$ is again a profile.
Thus in the niveau $2$ case the last line of the table should be conventionally interpreted as the usual symmetric difference $J \triangle \{i+1,i+f+1\}$.

Thanks to the discussion in the previous two paragraphs, it remains to establish the following.
\begin{itemize}
    \item For an operation $\mu_{j}$: that $\cL(\mu_{j}(\sigma),J) \in \cL(\sigma')$ when $j \in J, j+1 \not\in J$.

        \item For an operation $\theta_{j}$: that $\cL(\theta_{j}(\sigma),J) \in \cL(\sigma')$ when $j \in J, j+1 \in J$.

            \item For an operation $\nu _{j'}$: that $\cL(\nu_{j'}(\sigma),J) \in \cL(\sigma')$ when $j' \in J, j'+1 \in J$.   
\end{itemize}

We also note the following lemma.

\begin{lemma}\label{lem:two-values} Fix a pinning of $\sigma,\sigma'$.
As $J$ varies over $\cP(\sigma,\sigma')$ the product $\eps_i^{J} x^J_i$ takes at most two values for each $i$.
\end{lemma}

\begin{proof} The exceptional $p=3$ case can be checked directly, so we can assume we are not in that case.
Multiplying the equation $\eps_i^{K_J} u'_i - \eps_i^J u_i = p x^J_{i-1} - x^J_i $ by $\eps_i^J$, reducing mod $p$, and rearranging gives $u_i \pm u'_i \equiv \eps_i^J x^J_i \pmod{p}$. Thus $\eps_i^J x^J_i $ takes at most two different values mod $p$. But $-2,0,2$ are distinct mod $p$ for all odd $p$, so in fact $\eps_i^J x^J_i $ takes at most two different values in $\Z$.
\end{proof}

We are now ready to proceed with the four cases.

\begin{prop}\label{prop:Pa}
Suppose $\Zss(\sigma) \subseteq \Zss(\sigma')$, equipped with a pinning. Suppose that $x_{j-1} \le 0$, $u_j > 1$, and that, treating $j$ as an integer, there exists $j' \ge j$ with 
\[ x_j,x_{j+1},\ldots,x_{j'} > 0, \quad x_{j'+1} = 0, \quad u_{j+1} = \cdots = u_{j'+1} = 0 .\]
Then $\Zss(\mu_{j}(\sigma)) \subseteq \Zss(\sigma')$. 

In particular this conclusion holds in case $(P.a)$ if $\Zss(\sigma) \subseteq \Zss(\sigma')$ is equipped with a minimal pinning.
\end{prop}

\begin{proof}
 If $(P.a)$ holds and $\Zss(\sigma) \subseteq \Zss(\sigma')$ is equipped with a minimal pinning, then case $(P)$ of Theorem~\ref{thm:preliminaries} establishes the existence of $j'$ such that  $x_j,x_{j+1},\ldots,x_{j'} > 0$, $x_{j'+1} = 0$, and $u_{j+1} = \cdots = u_{j'+1} = 0$.  Furthermore $u_{j} > 1$ by Lemma~\ref{lem:minimal-reductions}($i$). Thus the final sentence of the Proposition follows from the first part. 
 
 Note  that since $x_i \neq 2$ for all $i$, we have $x_j = x_{j+1} = \cdots = x_{j'} = 1$. The last sentence of Remark~\ref{rem:key-critical} shows that~$\sigma'$ is non-critical. In particular $\cP(\sigma,\sigma')$ contains every profile, and $\cL(\sigma) \subseteq \cL(\sigma')$.
 
Put $T=\{j+1,j+2,\ldots,j'+1\}$.
Lemma~\ref{lem:carry-equation} applied at $i \in T$ gives
\begin{equation}\label{eq:Pa-carry} \eps_i^{K_J} u'_i = px^J_{i-1} - x^J_{i}.
\end{equation}
Taking $J = \varnothing$, our hypotheses on the $x_i$'s give
\[ 
u'_i = 
\begin{cases}
p-1 & i \in T \smallsetminus\{j'+1\}\\
p & i = j'+1
\end{cases}
\]
 and $K \cap T = \varnothing$. Feeding this back into \eqref{eq:Pa-carry}, we obtain
 \[ x_{i-1}^J = x_{i}^J = \eps_i^{K_J} \]
for $i \in T \smallsetminus \{j'+1\}$, while $x_{j'+1}^J = 0$ and $\eps_{j'+1}^{K_J} = x_{j'}^J$. That is, either
\[ x^J_{j} = \cdots = x^J_{j'} = 1, \quad T \cap K_J = \varnothing \qquad \text{or} \qquad x^J_{j} = \cdots = x^J_{j'} = -1, \quad T \subset K_J, \]
along with $x^J_{j'+1} = 0$.

 By the discussion immediately before the statement of the Proposition, it suffices to prove that $\cL(\mu_{j}(\sigma),J) \in \cL(\sigma')$ for profiles $J$ such that $j\in J$ and $j+1\notin J$. Fix such a profile. We have
 \[ \cL(\mu_{j}(\sigma), J) = \cL(\sigma,J) -
 \begin{cases}
p^{-j} & J \text{ of niveau 1}\\
(1-q)p^{-j} & J \text{ of niveau 2}.
 \end{cases}
 \]
Set $\lambda = 1$ if $J$ has niveau $1$ and $\lambda = 1-q$ if $J$ has niveau $2$, so that we can uniformly write $\cL(\mu_{j}(\sigma), J) = \cL(\sigma,J) - \lambda \cdot  p^{-j}$. This is an equation in $\Z/(q-1)\Z$ if $J$ has niveau $1$, and in $\Z/(q^2-1)\Z$ if $J$ has niveau $2$.
 
 Observing that $\sum_{i=j+1}^{j'+1} u'_i p^{-i}=p^{-j}$, if $T \subset K_J$ then 
\[ \cL(\mu_{j}(\sigma),J) = \cL(\sigma',K_J \triangle T) \]
as desired, so it remains to consider the case $T \cap K_J = \varnothing$, whence $x_j^J = 1$.

 Lemma~\ref{lem:alt} at the index $j$ for the profile $J$  gives the alternatives
\begin{equation}\label{eq:Pa-alt}
    s_{j-1}^J = s_j^J = 0 \qquad \text{or} \qquad u_j = pd_{j-1}^J - d_j^J .
\end{equation}
Since $x_j = x_j^J = 1$, we have $s_j^J= 1$ and $d_j^J = 0$. This rules out the first alternative, and the second alternative becomes $u_j = pd_{j-1}^J$. Since $u_j > 1$ we get $u_j = p$ and $d_{j-1}^J = 1$.

Now either $x_{j-1} = -1$ or $x_{j-1} = 0$. (The first part of Lemma~\ref{lem:minimal-reductions} rules out $x_{j-1} = -2$, since $x_j = 1$.) If $x_{j-1} = -1$, then $x_{j-1}^J = 1$ and Lemma~\ref{lem:carry-equation} for the profile $J$ gives
\[ \eps_j^{K_J} u'_j + u_j = px_{j-1}^J - x_j^J = p-1.\]
Since $u_j = p$ we have $u'_j = 1$ and $j \in K_J$. Then 
\[ \cL(\sigma',K_J \triangle\{j\}) = \cL(\sigma',K_J) - \lambda \cdot p^{-j} = \cL(\sigma,J) - \lambda \cdot p^{-j} = \cL(\mu_{j}(\sigma),J)\] as desired.

If instead $x_{j-1} = 0$, then $x_{j-1}^J = 2$, and Lemma~\ref{lem:carry-equation} for the profile $J$ gives
$\eps_j^{K_J} u'_j + u_j = 2p-1$, so that $j \not\in K_J$ and $u'_j = p-1$.  Now consider Lemma~\ref{lem:carry-equation} applied to the profile $J^{-} := J \triangle \{j\}$. We find that
\[ \eps_j^{K_{J^{-}}} (p-1) - p = p x_{j-1}^{J^{-}} - x_j^{J^{-}}.  \]
If $\eps_j^{K_{J^{-}}} = -1$ then $x_{j-1}^{J^{-}} = -2$. But since $J^{-}$ and $J$ either both contain $j-1$ or both do not, this would give $\{ x_{j-1},  \eps_{j-1}^{J^{-}} x_{j-1}^{J^{-}}, \eps_{j-1}^J x_{j-1}^J \} = \{-2,0,2\}$, in contradiction to Lemma~\ref{lem:two-values}. It follows that $\eps_j^{K_{J^{-}}} = 1$, and $j \not\in K_{J^{-}}$. Finally
\begin{align*}
\cL(\sigma', K_{J^{-}} \triangle \{j\})&  = \cL(\sigma',K_{J^{-}}) + \lambda \cdot (p-1) p^{-j} \\ 
& = \cL(\sigma,J^{-}) + \lambda \cdot (p-1) p^{-j} \\
& = \cL(\sigma,J) - \lambda \cdot p^{-j} \\
& = \cL(\mu_{j}(\sigma),J) 
\end{align*}
and this completes the proof. 
\end{proof}

\begin{prop}\label{prop:Pb}
Suppose $\Zss(\sigma) \subseteq \Zss(\sigma')$, equipped with a pinning.
Suppose that $x_{j-1}\le 0$, $u_j=0$, and that, treating $j$ as an
integer, there exists $j'\ge j$ with
\[
 x_j,x_{j+1},\ldots,x_{j'}>0,\qquad x_{j'+1}=0,\qquad
 u_{j+1}=\cdots=u_{j'+1}=0.
\]
Then $\Zss(\theta_{j}(\sigma))\subseteq\Zss(\sigma').$

In particular this conclusion holds in case {\upshape$(P.b)$} if
$\Zss(\sigma)\subseteq\Zss(\sigma')$ is equipped with a minimal pinning.
\end{prop}

\begin{proof}
If $(P.b)$ holds and $\Zss(\sigma) \subseteq \Zss(\sigma')$ is equipped with a minimal pinning, then case $(P)$ of Theorem~\ref{thm:preliminaries} establishes the existence of $j'$ such that  $x_j,x_{j+1},\ldots,x_{j'} > 0$, $x_{j'+1} = 0$, and $u_{j+1} = \cdots = u_{j'+1} = 0$. Thus the final sentence follows.

The last sentence of Remark~\ref{rem:key-critical} shows that~$\sigma'$ is non-critical. In particular $\cP(\sigma,\sigma')$ contains every profile, and $\cL(\sigma) \subseteq \cL(\sigma')$.

Since $x_i \neq 2$ for all $i$, we have $x_j=x_{j+1}=\cdots=x_{j'}=1.$ Now $\eps_j^K u'_j=p x_{j-1}-1$, and since $x_{j-1}\le0$, this forces
\[
 x_{j-1}=0,\qquad j\in K,\qquad u'_j=1.
\]
At the other end of the interval we have
\[
 \eps_{j'+1}^K u'_{j'+1}
 =p x_{j'}-x_{j'+1}=p,
\]
and therefore
\[
 j'+1\notin K,\qquad u'_{j'+1}=p.
\]

By the discussion immediately before Proposition~\ref{prop:Pa}, it
suffices to prove that
$\cL(\theta_{j}(\sigma),J)\in\cL(\sigma')$ for profiles $J$ such that
$j,j+1\in J$. Fix such a profile, and set $\lambda=1$ if $J$ has
niveau $1$ and $\lambda=1-q$ if $J$ has niveau $2$. Then
\[
 \cL(\theta_{j}(\sigma),J)
 =\cL(\sigma,J)+\lambda \cdot p^{-j}.
\]
Lemma~\ref{lem:carry-equation} at the index $j$ gives
$
 \eps_j^{K_J}=p x_{j-1}^J-x_j^J,
$
so that  \[ x_{j-1}^J=0,\qquad x_j^J\in\{-1,1\}.\]
Similarly at $j'+1$ we have
$
 \eps_{j'+1}^{K_J}p
 =p x_{j'}^J-x_{j'+1}^J$ and therefore $x_{j'+1}^J=0$.

Suppose first that $x_j^J=1$. Put $
 T=\{j+1,j+2,\ldots,j'+1\}. $
Lemma~\ref{lem:carry-equation} for $J$ at an index $i \in T$ gives $\eps_i^{K_J} u'_i = p x_{i-1}^J - x_{i}^J$, since $u_i = 0$ for $i \in T$. Multiplying by $p^{-i}$ and  summing gives
\[
 \sum_{i\in T}\eps_i^{K_J}u'_i p^{-i}
 =x_j^Jp^{-j}-x_{j'+1}^Jp^{-(j'+1)}
 =p^{-j}.
\]
Consequently
\[
 \cL(\sigma',K_J\triangle T)
 =\cL(\sigma',K_J)+\lambda \cdot  p^{-j}
 =\cL(\theta_{j}(\sigma),J)
\]
and the Proposition follows in this case.

If instead $x_j^J=-1$, we have $\eps_j^{K_J} = -x_j^J = 1$ and
$j\notin K_J$. Since $u'_j=1$, we obtain
\[
 \cL(\sigma',K_J\triangle\{j\})
 =\cL(\sigma',K_J)+\lambda \cdot p^{-j}
 =\cL(\theta_{j}(\sigma),J).
\]
This completes the proof.
\end{proof}

\begin{prop}\label{prop:Na}
Suppose $\Zss(\sigma) \subseteq \Zss(\sigma')$, equipped with a pinning.
Suppose $j' \ge j$ are integers such that
\[
 x_{j-1}=0,\qquad x_j,x_{j+1},\ldots,x_{j'}<0,\qquad
 x_{j'+1}\le0,
\]
and
\[
 u_j=u_{j+1}=\cdots=u_{j'}=0,\qquad
 0<u_{j'+1}<p.
\]
Then  $\Zss(\nu_{j'}(\sigma))\subseteq\Zss(\sigma').$

In particular this conclusion holds in case {\upshape$(N.a)$} if
$\Zss(\sigma)\subseteq\Zss(\sigma')$ is equipped with a minimal pinning.
\end{prop}

\begin{proof}
The final sentence follows because
Lemma~\ref{lem:minimal-reductions}$(ii)$ rules out $u_{j'+1}=p$ in
case {\upshape$(N.a)$}. The last sentence of Remark~\ref{rem:key-critical} shows that~$\sigma'$ is non-critical. In particular $\cP(\sigma,\sigma')$ contains every profile, and $\cL(\sigma) \subseteq \cL(\sigma')$.


If $u_{i+1} = 0$ and $x_{i} = -2$ then Lemma~\ref{lem:carry-equation} at $i+1$ gives $u'_{i+1} = |px_{i} - x_{i+1}| \ge 2p - 2 > p$, a contradiction. Hence $x_i = -1$ for $j \le i < j'$ while $x_{j'} \in \{-1,-2\}$.

For each $j \le i \le j'$, since $u_i = 0$, Lemma~\ref{lem:alt} for any profile $J$ gives either $s_i^{J} = s_{i-1}^{J} = 0$ or $d_i^{J} = d_{i-1}^{J} = 0$ (but not both, since $x_i \neq 0$). Thus either $s_i^J = 0$ for all $j-1 \le i \le j'$ or $d_i^J = 0$ for all $j-1 \le i \le j'$. In any case we have $x_i^J = \pm x_i$ for all $j-1 \le i \le j'$.

By the discussion preceding Proposition~\ref{prop:Pa}, it suffices to
consider profiles $J$ such that $j',j'+1\in J$. Fix such a profile, and
set
\[
 J^{-}=J\triangle\{j'+1\}.
\]
Thus $j'+1\notin J^{-}$. If $\lambda=1$ in niveau $1$ and
$\lambda=1-q$ in niveau $2$, then a direct calculation gives
\begin{equation}\label{eq:Na-bridge}
 \cL(\nu_{j'}(\sigma),J)
 =\cL(\sigma,J^{-})+\lambda \cdot p^{-j'}.
\end{equation}
For each $i$ we have $\eps_i^{K_{J^-}} u'_i - \eps_i^{J^-} u_i = px_{i-1}^{J^-} - x^{J^-}_i$. Multiplying by $p^{-i}$, summing from $j$ to $j'$, and using $u_i = 0$ in that range gives 
\[
 \sum_{i=j}^{j'}\eps_i^{K_{J^{-}}}u'_i p^{-i}
 =-x_{j'}^{J^{-}}p^{-j'}.
\]
Summing instead from $j$ to $j'+1$ gives
\[
 \sum_{i=j}^{j'+1}\eps_i^{K_{J^{-}}}u'_i p^{-i}
 =(u_{j'+1} -x_{j'+1}^{J^{-}})p^{-j'-1}.
\]
Thus if $x_{j'}^{J^-} = -1$ then 
\[ \cL(\sigma',K_{J^{-}} \triangle  \{j,\ldots,j'\}) = \cL(\sigma',K_{J^{-}}) + \lambda \cdot p^{-j'} = \cL(\sigma,J^-) + \lambda \cdot p^{-j'} = \cL(\nu_{j'}(\sigma),J), \]
while if  $u_{j'+1} - x_{j'+1}^{J^{-}} = p$ then similarly
\[ \cL(\sigma',K_{J^{-}} \triangle  \{j,\ldots,j'+1\}) =  \cL(\nu_{j'}(\sigma),J). \]

We can therefore suppose for the rest of the argument that neither  $x_{j'}^{J^-} = -1$  nor  $u_{j'+1} - x_{j'+1}^{J^{-}} = p$.
Lemma~\ref{lem:alt} for $J^-$ at the index $j'+1$ gives the alternatives
\begin{equation}
    \label{eq:Na-eq}
    -u_{j'+1} = ps^{J^-}_{j'} - s^{J^-}_{j'+1} \qquad \text{or} \qquad d^{J^-}_{j'} = d^{J^-}_{j'+1} = 0 .
\end{equation}

In the second alternative $x_{j'}^{J^-} = x_{j'}$. By hypothesis this is not $-1$, so it is $-2$. Now Lemma~\ref{lem:carry-equation} gives $\eps_{j'+1}^{K_{J^-}} u'_{j'+1}  = -2p + (u_{j'+1} - x_{j'+1})$. Since by hypothesis $u_{j'+1} - x_{j'+1} \neq p$, but also $u_{j'+1} < p$, the only possibility is $u_{j'+1} = u'_{j'+1} = p-1$ and $x_{j'+1} = -2$. We claim this is impossible. To see this let $J'$ be any profile containing $j'+2$ but not $j'+1$. Lemma~\ref{lem:carry-equation} at the index $j'+1$ gives
\[ (\eps^{K_{J'}}_{j'+1} - 1) (p-1) = px_{j'}^{J'} - x^{J'}_{j'+1}. \]
Using our  observation in the third paragraph of the proof that $x_{j'}^{J'} = \pm x_{j'} = \pm 2$, the preceding equation has exactly one solution, with $\eps^{K_{J'}}_{j'+1} = -1$ and $x_{j'}^{J'} =  x^{J'}_{j'+1} = -2$. But then Lemma~\ref{lem:carry-equation} at the index $j'+2$ gives
\[ \eps^{K_{J'}}_{j'+2}u'_{j'+2} + u_{j'+2} = -2p - x^{J'}_{j'+2},  \]
and this is a contradiction because the left-hand side is at least $-p$ while the right-hand side is at most $-2p+2$. 

So it is the first alternative in \eqref{eq:Na-eq} that must hold. If $s_{j'}^{J^-} = -1$, then by the observation in the third paragraph of the proof we must have $d_{j'}^{J^-} = 0$. But together these  give $x_{j'}^{J^-} = -1$, contradicting our running hypothesis. The last remaining possibility is that  $s_{j'}^{J^-} = 0$ and $u_{j'+1} = s_{j'+1}^{J^-} = 1$. Since $x_{j'+1} = 0$ we have $x_{j'+1}^{J^-} = 2$.  Lemma~\ref{lem:carry-equation} at the index $j'+1$ is 
\[ \eps_{j'+1}^{K_{J^-}} u'_{j'+1} - 1 = p x^{J^{-}}_{j'}  - 2 \]
and so we find that $x_{j'}^{J^{-}} = 1$, $x_{j'} = -1$, $u'_{j'+1} = p-1$, and $j'+1 \not\in K_{J^-}$.

If $j'+2\notin J$, then
$j'+2\notin J^{-}$, and Lemma~\ref{lem:carry-equation} at $j'+2$ for the profile $J^-$ gives
\[
 \eps_{j'+2}^{K_{J^{-}}}u'_{j'+2}
   -u_{j'+2}+x_{j'+2}^{J^{-}}=2p.
\] 
The left-hand side is at most $p+2<2p$, a contradiction. It follows that $j'+2\in J$.  Applying the same lemma for the profile $J$ then gives 
\[
 \eps_{j'+2}^{K_J }u'_{j'+2}
    + u_{j'+2}+x_{j'+2}^{J} = p x_{j'+1}^J .
\] 
This is impossible if $x_{j'+1}^J = -2$, so $x_{j'+1}^J \in \{0,2\}$. But in Lemma~\ref{lem:two-values} applied at the index $j'+1 \in J$, the three values for the profiles $\varnothing,J^{-},J$ are $0,2,-x_{j'+1}^J$ respectively. The Lemma thus rules out $x_{j'+1}^J = 2$, and we deduce that $x_{j'+1}^J = 0$.  Lemma~\ref{lem:carry-equation} at $j'+1$ for the profile $J$ then gives $\eps_{j'+1}^{K_J} (p-1) + 1 = p x^J_{j'}$, and we conclude that $j'+1 \not\in K_J$.

Finally, a pair of direct calculations using $u'_{j'+1} = p-1$ and $u_{j'+1} = 1$ gives 
\[
 \cL(\sigma',K_{J} \triangle\{j'+1\})
 =\cL(\sigma,J)
   +\lambda \cdot (p-1)p^{-(j'+1)} =  \cL(\nu_{j'}(\sigma),J)
\]
and the proof is complete.
\end{proof}

Before addressing the fourth and final case, we note the following lemma.

\begin{lemma}\label{lem:Nb-zero-string}
Suppose $\Zss(\sigma)\subseteq\Zss(\sigma')$, equipped with a pinning.
Suppose $x_i\le0$ for all $i$, and
\[
 x_{j-1}=0,\qquad x_j<0,\qquad u_j>0,\qquad u_{j+1}=0.
\]
Treating $j$ as an integer, let $r>j$ be minimal such that $x_r=0$.
Then
\[
 u_{j+1}=u_{j+2}=\cdots=u_r=0.
\]
\end{lemma}

\begin{proof} If $\sigma'$ is critical then $\sigma$ is scalar and the result is immediate, so we may assume that $\sigma'$ is non-critical.
Lemma~\ref{lem:carry-equation} at $j+1$ rules out $x_j=-2$, and hence $x_j=-1$.
Lemma~\ref{lem:carry-equation} at $j$ then gives
\[
 j\notin K,\qquad u'_j=u_j+1,\qquad 0<u_j<p.
\]

Suppose for the sake of contradiction that some $u_i$ with $j<i\le r$ is
positive, and let $m$ be the least such index. If
$m+1\not\equiv j\pmod f$, let $J=\{j\}$.
If $m+1\equiv j\pmod f$, take $J$ to be the niveau $2$ profile satisfying $ J \cap\{j,j+1,\ldots,j+f-1\}=\{j\}.$
Since $\sigma'$ is non-critical, in either case $J \in
\cP(\sigma,\sigma')$. 

Lemma~\ref{lem:carry-equation} at the index $j$ for the profile $J$ is
\[
 \eps_j^{K_{J}}(u_j+1)+u_j
   =p x_{j-1}^{J}-x_j^{J}.
\]
Parity and the bounds $0 < u_j < p$ force $x_j^{J}=1$. Hence
$s_j^{J}=0$ and $d_j^{J}=1$.

For $j<i<m$, we have $u_i=0$ and $i\notin J$. The first alternative
of Lemma~\ref{lem:alt} gives $s_{i-1}^{J} = s^J_i = 0$ while the second gives $d_{i-1}^J = d_i^J = 0$. Furthermore $s_i^J$, $d_i^J$ are not both $0$ for $i$ in this range because $x_i < 0$. Since $d_j^J = 1$, it follows iteratively that the first alternative must be the one that holds for all $i$ in this range, and
\[
 s_j^{J}=s_{j+1}^{J}=\cdots=s_{m-1}^{J}=0.
\]
Applying Lemma~\ref{lem:alt} at $m\notin J$, its second alternative is
again impossible, while the first gives $u_m=s_m^{J}$.

If $m<r$, then $x_m<0$, and therefore $s_m^{J}\le0$, a contradiction.
Thus $m=r$. Since $x_r=0$, the equality
$u_r=s_r^{J}>0$ forces $u_r=s_r^{J}=1$ and $x_r^J = 2$.
But $r+1 = m+1 \notin J$, so Lemma~\ref{lem:carry-equation} at $r+1$ gives
\[
 \eps_{r+1}^{K_{J}}u'_{r+1}-u_{r+1}
   =2p-x_{r+1}^{J}.
\]
The left-hand side is at most $p$, whereas the right-hand side is at
least $2p-2>p$. This contradiction proves the lemma.
\end{proof}

\begin{prop}\label{prop:Nb}
Suppose $\Zss(\sigma) \subseteq \Zss(\sigma')$, equipped with a pinning.
Suppose $x_i\le0$ for all $i$, and that
\[
 x_{j-1}=0,\qquad x_j<0,\qquad u_{j+1}=0.
\]
Suppose moreover that either $u_j>0$; or $u_i=0$ for all $i$; or else
$u_j=0$, not all $u_i$ are zero, and, treating $j$ as an integer
 and writing $j'\ge j$ for the least integer such that
$u_{j'+1}>0$, there exists $r$ with
\[
 j\le r\le j',\qquad x_r=0.
\]
Then $\Zss(\theta_{j}(\sigma))\subseteq\Zss(\sigma').$

In particular this conclusion holds in case {\upshape$(N.b)$} if
$\Zss(\sigma)\subseteq\Zss(\sigma')$ is equipped with a minimal pinning.
\end{prop}

\begin{proof}
The final sentence follows from Theorem~\ref{thm:preliminaries}, which
gives $u_{j+1}=0$ in case {\upshape$(N.b)$}, together with the
description of the failure of the additional condition in
{\upshape$(N.a)$}. The last sentence of Remark~\ref{rem:key-critical} shows that~$\sigma'$ is non-critical. In particular $\cP(\sigma,\sigma')$ contains every profile, and $\cL(\sigma) \subseteq \cL(\sigma')$.

We first claim that there exists an integer $r>j$ such that
\begin{equation}\label{eq:Nb-zero-interval}
 x_j,x_{j+1},\ldots,x_{r-1}<0,\qquad x_r=0,\qquad
 u_{j+1}=\cdots=u_r=0.
\end{equation}
If $u_j>0$ or if $u_i = 0$ for all $i$, take $r>j$ minimal such that $x_r=0$ (and in the former case, apply
Lemma~\ref{lem:Nb-zero-string}). In the remaining case take $r$ to
be the least index in $[j,j']$ for which $x_r=0$;\ the definition of
$j'$ gives $u_{j+1}=\cdots=u_r=0$. 

Lemma~\ref{lem:carry-equation} at $j+1$, together with $u_{j+1}=0$, rules out
$x_j=-2$, so $x_j=-1$. Lemma~\ref{lem:carry-equation} at $j$ then gives
\[
 j\notin K,\qquad u'_j=u_j+1,\qquad u_j<p.
\]
Lemma~\ref{lem:carry-equation} at  $r$ gives 
$ \eps_r^K u'_r=p x_{r-1}$, and since $x_{r-1} < 0$  we have
\[
 x_{r-1}=-1,\qquad r\in K,\qquad u'_r=p.
\]

By the discussion preceding Proposition~\ref{prop:Pa}, it suffices to
consider profiles $J$ such that $j,j+1\in J$. Fix such a profile, and
set $\lambda=1$ in niveau $1$ and $\lambda=1-q$ in niveau $2$. We have
\[ \cL(\theta_{j}(\sigma),J)
 =\cL(\sigma,J)+\lambda \cdot p^{-j}.
\]
Lemma~\ref{lem:carry-equation} at $j$ is
\[ \eps_j^{K_J}(u_j+1)+u_j
   =p x_{j-1}^J-x_j^J.
\]
If $u_j>0$ then since $u_j \neq p$ we have $x_j^J=1$. If $u_j=0$ then instead $x_{j-1}^J=0$ and $x_j^J\in\{-1,1\}$. Note that $x_j^J=-1$ forces $j\notin K_J$. At $r$ we have $\eps_r^{K_J}p=p x_{r-1}^J-x_r^J$,
and therefore $x_r^J=0$.

Suppose first that $x_j^J=1$. Put
\[
 T=\{j+1,j+2,\ldots,r\}.
\]
Taking Lemma~\ref{lem:carry-equation} for $i \in T$, multiplying by $p^{-i}$, and summing gives
\[
 \sum_{i\in T}\eps_i^{K_J}u'_i p^{-i}
 =x_j^Jp^{-j}-x_r^Jp^{-r}
 =p^{-j},
\] where we use 
\eqref{eq:Nb-zero-interval} to see that $u_i = 0$ in this range.
It follows that
\[
 \cL(\sigma',K_J\triangle T)
 =\cL(\sigma,J)+\lambda \cdot p^{-j}
 =\cL(\theta_{j}(\sigma),J).
\]
The other possibility is that $u_j=0$ and $x_j^J=-1$.
In this case $j\notin K_J$ and $u'_j=1$, so again
\[
 \cL(\sigma',K_J\triangle\{j\})
 =\cL(\sigma,J)+\lambda \cdot p^{-j}
 =\cL(\theta_{j}(\sigma),J)
\]
and we are done. 
\end{proof}

\section{Non-invertible operations}

The final ingredient we need for the proof of the main theorem is that non-invertible valid operations $f_j$ give strict inclusions. Note that this is clear for an inclusion of stacks $\cZ(\sigma) \subseteq \cZ(f_j(\sigma))$, for dimension reasons, but requires justification for sets of semisimple points.

\begin{prop}\label{prop:strict-operations}
Suppose that \(f_j\in\{\mu_j,\theta_j,\nu_j\}\) is valid and
non-invertible on \(\sigma\). Then
\[
 \Zss(\sigma)\subsetneq\Zss(f_j(\sigma)).
\]
More precisely:
\begin{itemize}
    \item if $f_j = \mu_j$ then $\rhobar(\mu_j(\sigma),\{j+1\}) \not\in \Zss(\sigma)$;
    \item if $f_j = \theta_j$ then $\rhobar(\theta_j(\sigma),\varnothing) \not\in \Zss(\sigma)$;
    \item if $f_j = \nu_j$ then $\rhobar(\nu_j(\sigma),\varnothing) \not\in \Zss(\sigma)$.
\end{itemize}
\end{prop}

\begin{proof}
Write \(\sigma=(u_i;t)\). Since \(f_j\) is valid, \(u_{j+1}=0\) if $f_j \in \{\mu_j,\theta_j\}$ and $u_j = 0$ if $f_j = \nu_j$. In all cases
\(\sigma\) is non-critical. By Remark~\ref{rem:inclusion-critical-rem},
it is enough to exhibit an element of
$\cL(f_j(\sigma))\smallsetminus\cL(\sigma)$. Without loss of generality take $t = 0$.

Recall that if $\sum_i p^{-i} y_i = 0$ in $\Z/(q-1)\Z$  then
there are unique integers \(z_i\) such that
$
 y_i=pz_{i-1}-z_i.
$ If moreover $y_i \in [-p,p]$ for all $i$, then the same argument as in Corollary~\ref{cor:scalar-carry} gives $|z_i| \le 1$ for all $i$.

Suppose first that \(f_j=\mu_j\). We have $\cL(\mu_j(\sigma),\{j+1\})=p^{-j}.$
If this belonged to \(\cL(\sigma)\), there would exist a niveau \(1\)
profile \(I\) such that
$
 \sum_{i\in I}p^{-i}u_i=p^{-j}.
$
Apply the observation in the second paragraph of the proof to
\[
 y_i=\delta_I(i)u_i-\delta_{\{j\}}(i).
\]
Since \(y_{j+1}=0\), we obtain \(z_{j}=z_{j+1}=0\), and therefore
\(y_{j}=pz_{j-1}\in\{-p,0,p\}\). On the other hand since $u_{j} > 1$ by non-invertibility, we have
\[
 y_{j}=
 \begin{cases}
  u_{j}-1\in[1,p-1]&\text{if }j\in I,\\
  -1&\text{if }j\notin I,
 \end{cases}
\]
a contradiction.

Suppose next that \(f_j=\theta_j\). We have
$
 \cL(\theta_j(\sigma),\varnothing)=-p^{-j}.
 $
If this belonged to \(\cL(\sigma)\), there would be a niveau \(1\)
profile \(I\) such that
$
 \sum_{i\in I}p^{-i}u_i=-p^{-j}.
$
Apply the observation in the second paragraph to
\[
 y_i=\delta_I(i)u_i+\delta_{\{j\}}(i).
\]
Validity of \(\theta_j\) gives \(0\le y_i\le p\) for all $i$. Again \(y_{j+1}=0\), so
\(z_{j}=z_{j+1}=0\). Since \(y_{j}>0\), the equation $y_{j} = pz_{j-1} - z_{j}$
gives \(z_{j-1}=1\). But whenever \(z_i=1\), the inequality
\(y_i\ge0\) forces \(z_{i-1}=1\). Iterating cyclically eventually gives
\(z_{j+1}=1\), a contradiction.

Finally suppose that \(f_j=\nu_j\). We have
$
 \cL(\nu_j(\sigma),\varnothing)
 =-p^{-j}+u_{j+1} \,p^{-(j+1)}.
$
If this belonged to \(\cL(\sigma)\), there would be a niveau \(1\)
profile \(I\) such that
\[
 \sum_{i\in I}p^{-i}u_i=-p^{-j}+u_{j+1} \,p^{-(j+1)}.
\]
Apply the observation in the second paragraph to
\[
 y_i=\delta_I(i)u_i+\delta_{\{j\}}(i)
       -u_{j+1} \delta_{\{j+1\}}(i).
\]
Here \(y_j=1\), \(y_{j+1}\ge-(p-1)\) because $u_{j+1}\neq p$ by non-invertibility,  and \(y_i\ge0\) for
\(i\ne j+1\). The equation $y_j = pz_{j-1} - z_j$ forces
$z_{j-1} = 0$ and $z_j = -1$. The equation $y_{j+1} = -p - z_{j+1}$, together with \(y_{j+1}\ge-(p-1)\), 
forces \(z_{j+1}=-1\). But now
$ y_{j+2}=-p-z_{j+2}<0$, 
contrary to \(y_{j+2}\ge0\). This completes the proof.
\end{proof}

\section{The main theorems}

We are now ready to prove Theorem~\ref{thm:Zss}, thus also establishing Theorems~\ref{thm:equivalence-Z},~\ref{thm:simple-inclusions}, and~\ref{thm:non-simple-inclusions} as a consequence of Proposition~\ref{prop:reduction}. For convenience we restate the theorem here, in the language of phantom Serre weights.

\begin{thm}\label{thm:Zss-in-phantom}
   We have the following.
 \begin{enumerate}
     \item   We have $\Zss(\sigma) = \Zss(\sigma')$ if and only if  either $\sigma \sim \sigma'$ or else there exists $t$ such that either $\sigma = (\BT;t)$ and $\sigma' = (\St;t)$ or vice-versa.

\item Suppose $\Zss \subsetneq \Zss'$ is a simple inclusion in $\rZss$. Then either:
\begin{itemize}
    \item there exists $\sigma$ and a valid operation $f_j \in \{\mu_j,\theta_j,\nu_j\}$ such that $\Zss = \Zss(\sigma)$ and $\Zss' = \Zss(f_j(\sigma))$, or

    \item we have $\Zss = \Zss((\ur;t))$ and $\Zss' = \Zss((\crit;t))$ for some $t$.
\end{itemize}

\item Each of the inclusions listed in $(ii)$ is simple except for the ones of the form 
$\Zss(\sigma) \subsetneq \Zss(f_j(\sigma))$ with $f_j \in \{\theta_j,\nu_j\}$, $(u_{j},u_{j+1}) = (0,0)$, and either $u_{j-1} = 1$ or $u_{j+2} = p$.
\end{enumerate}
\end{thm}

\begin{proof}
$(i)$ The `if' direction is clear. To check the `only-if' direction, suppose that $\Zss(\sigma) = \Zss(\sigma')$. In particular $\Zss(\sigma) \subseteq \Zss(\sigma')$. Choose $\widetilde{\sigma} \sim \sigma$ such that $\Zss(\widetilde\sigma) \subseteq \Zss(\sigma')$ can be equipped with a minimal pinning. Consider what  Theorem~\ref{thm:key} says about this inclusion. All of the six bullet points listed in the theorem are proper inclusions, thanks to Remark~\ref{rem:crucial} and Proposition~\ref{prop:strict-operations}. Since the inclusion $\Zss(\widetilde\sigma) \subseteq \Zss(\sigma')$ is not proper, the conclusion of the theorem cannot hold, and so either $\widetilde\sigma = \sigma'$ or $\{\widetilde{\sigma},\sigma'\} = \{(\BT;t),(\St;t)\}$. In either case we are done. In the first case we have $\sigma \sim \widetilde\sigma = \sigma'$;\ in the second case, since neither $(\BT;t)$ nor $(\St;t)$ is the source or target of any of the invertible operations, we have $\sigma = \widetilde\sigma$, and $\{\sigma,\sigma'\} = \{(\BT;t),(\St;t)\}$.

$(ii)$ This part is essentially immediate from Theorem~\ref{thm:key} and Proposition~\ref{prop:strict-operations}. Write $\Zss \subsetneq \Zss'$ as $\Zss(\sigma) \subsetneq \Zss(\sigma')$ equipped with a minimal pinning. Since the inclusion is strict we have $\sigma \neq \sigma'$ and $\{\sigma,\sigma'\} \neq \{(\BT;t),(\St;t)\}$. Theorem~\ref{thm:key} tells us that either $\Zss = \Zss((\ur;t))$ and $\Zss' = \Zss((\crit;t))$ for some $t$, or else one of the four middle bullet points holds and we have
\[ \Zss(\sigma) \subseteq \Zss(f_j(\sigma)) \subseteq \Zss(\sigma') \]
for some valid and non-invertible $f_j$. Proposition~\ref{prop:strict-operations} tells us that the left-hand inclusion is strict, and the hypothesis that $\Zss \subsetneq \Zss'$ is simple implies $\Zss(f_j(\sigma)) = \Zss(\sigma')$.

$(iii)$ We already saw in the introduction, in the discussion following the statement of Theorem~\ref{thm:non-simple-inclusions}, that each of the listed exceptions is non-simple on the level of the stacks $\cZ$. For each listed exception, that discussion exhibits $f_j(\sigma) = f(g(h(\sigma)))$ where $f,g,h$ are valid operations, two of which are non-invertible and one of which is invertible. It follows from Proposition~\ref{prop:strict-operations} that two of the inclusions in 
\[ \Zss(\sigma) \subseteq \Zss(h(\sigma)) \subseteq \Zss(g(h(\sigma))) \subseteq \Zss(f(g(h(\sigma)))) = \Zss(f_j(\sigma)) \]
are strict, and so $\Zss(\sigma) \subseteq \Zss(f_j(\sigma))$ cannot be simple.

It remains to show that each of the inclusions that we claim to be simple is actually simple. The simplicity of $\Zss((\ur;t)) \subseteq \Zss((\crit;t))$ was already established by Proposition~\ref{prop:critical-target}. It remains to consider inclusions $\Zss(\sigma) \subsetneq \Zss(f_j(\sigma))$ for valid~$f_j$ not of the exceptional type.

Since parts $(i)$ and $(ii)$ have  been proved, observe that by Proposition~\ref{prop:reduction}, we already have Proposition~\ref{prop:reduction}$(i)$, Theorem~\ref{thm:equivalence-Z}, and Theorem~\ref{thm:simple-inclusions} at our disposal. If there is an intermediate inclusion $\Zss(\sigma) \subsetneq \Zss(\sigma') \subsetneq \Zss(\sigma'') = \Zss(f_j(\sigma))$, where we take $\sigma'' = f_j(\sigma)$ unless $f_j(\sigma) = (\St;t)$, in which case we take $\sigma'' = (\BT;t)$, then there is also a chain
\[ \cZ(\sigma) \subsetneq \cZ(\sigma') \subsetneq \cZ(\sigma'') . \]
Since each stack in this chain is irreducible, \[ \dim \cZ(f_j(\sigma)) = \dim \cZ(\sigma'') \ge \dim \cZ(\sigma) + 2.\]  
Write $\sigma = (u_i;t)$. Since the codimension of  $\cZ(\sigma)$  is given by $\#\{i : u_i = 0\}$, the operation must reduce the cardinality of this set by $2$, i.e., we are in the case where the operation is either $\theta_j$ or $\nu_{j}$ with $(u_{j-1},\ldots,u_{j+2}) = (u_{j-1},0,0,u_{j+2})$ and $u_{j-1} \neq 1$, $u_{j+2} \neq p$. Since in this situation $\theta_j = \nu_{j}$, it suffices to consider~$\theta_j$.\footnote{In the Lean formalization available at \url{https://github.com/davidsavitt/KLSposet}, this argument that appeals to the Emerton--Gee stacks is replaced by a combinatorial argument.}

Suppose, then, that we have a chain
\[ \Zss(\sigma) \subsetneq \Zss(\sigma') \subsetneq \Zss(\theta_j(\sigma)). \]
We can and do assume that $\Zss(\sigma) \subsetneq \Zss(\sigma')$ is simple. Since $\sigma'$ cannot be regular, because otherwise $\Zss(\sigma')$ would be maximal in $\rZss$, we have $\sigma' \sim f_{k}(\widetilde{\sigma})$ for some $\widetilde\sigma \sim \sigma$, some index $k$,  and a valid non-invertible operation $f_{k}$. We want to derive a contradiction.

Writing
$\widetilde{\sigma}=(\widetilde u_r;\widetilde t)$, we claim that we still have
\[
 \widetilde u_{j}=\widetilde u_{j+1}=0,\qquad
 \widetilde u_{j-1}\neq1,\qquad
 \widetilde u_{j+2}\neq p,
\]
and in addition we have
\begin{equation}\label{eq:equiv-local-nu}
 \theta_{j}(\widetilde{\sigma})\sim\theta_{j}(\sigma).
\end{equation}
Indeed, any invertible operation that alters one of $u_{j-1},\ldots,u_{j+2}$ can only be $\mu_{j-2}/\nu_{j-2}$, changing $u_{j-1}$ from $0$ to $p$ or vice-versa, or $\mu_{j+2}/\nu_{j+2}$,  changing $u_{j+2}$ from $1$ to $0$ or vice-versa. Thus any invertible operation on $\sigma$ preserves the given conditions on $(u_{j-1},\ldots,u_{j+2})$ and has support disjoint from that of $\theta_{j}$, hence commutes with $\theta_j$. The claim follows. 
Consequently, we may replace $\sigma$ with $\widetilde{\sigma}$, and $\sigma'$ with $f_{k}(\widetilde\sigma)$, so that our chain becomes
\[ \Zss(\sigma) \subsetneq \Zss(f_{k}(\sigma)) \subsetneq \Zss(\theta_j(\sigma)). \]
We will show in this situation that $f_{k} = \theta_j$ or $\nu_{j}$, which provides the necessary contradiction.

A direct calculation gives
\begin{equation}\label{eq:local-nu-coordinates}
 \{\cL(\theta_{j}(\sigma),J):J\subseteq\Z/f\Z\}
 =  \{\cL(\sigma,J) + \gamma \cdot p^{-j}:J\subseteq\Z/f\Z,\ \gamma \in \{-1,0,1\}\}.
\end{equation}
Take $J' = \{k+1\}$ if $f_{k} = \mu_{k}$ and $J' = \varnothing$ if $f_{k} = \theta_{k}$ or $\nu_{k}$.  Then
\[
\cL(f_{k}(\sigma),J') = t + \eta \quad \text{with} \quad
\eta := \begin{cases}
  p^{-k}&f_{k}=\mu_{k}\\
  -p^{-k} &f_{k}=\theta_{k}\\
  -p^{-k} +u_{k+1}p^{-(k+1)} &f_{k}=\nu_{k}
 \end{cases}
\]
and since $J'$ has niveau $1$, this value must be in the set in
\eqref{eq:local-nu-coordinates}. We can then write
\[
 \eta=\sum_{i\in J} u_i p^{-i} + \gamma \cdot p^{-j}
 \]
for some $J \subseteq\Z/f\Z$. Define $y_i$ by the formulas 
\[
\begin{split}
 y_i&=\delta_J(i)u_i+\gamma \cdot \delta_{\{j\}}(i)
             -\delta_{\{k\}}(i),\\
 y_i&=\delta_J(i)u_i+\gamma \cdot \delta_{\{j\}}(i)
             +\delta_{\{k\}}(i),\\
 y_i&=\delta_J(i)u_i+\gamma \cdot \delta_{\{j\}}(i)
             +\delta_{\{k\}}(i)
             -u_{k+1}\delta_{\{k+1\}}(i)
\end{split}
\]
in the three cases $f_{k} = \mu_{k}$, $\theta_{k}$, $\nu_{k}$ respectively. As usual we find $y_i = p z_{i-1} - z_i$ with $|z_i| \le 1$ for all $i$.

Suppose $k \neq j$ in the $\theta_{k}$ and $\nu_{k}$ cases. Note also that $k \neq j, j+1$ in the $\mu_{k}$ case, because $\mu_{j}$ and $\mu_{j+1}$ are not valid.

If $f_{k} = \mu_{k}$ then we have $y_{j+1}=0$, while in the other two cases
$y_{j+1} \in\{0,1\}$. Either way we have
$z_{j}=0$. Moreover, in all three cases $y_{j}=\gamma$, so the
equation $y_{j} = p z_{j-1} - z_{j}$ gives $\gamma = p z_{j-1}$. It follows that $\gamma = 0$. Thus in fact $\cL(f_{k}(\sigma),J') = \cL(\sigma,J) \in \cL(\sigma)$. 
Looking at the choice of $J'$ and comparing with the list in Proposition~\ref{prop:strict-operations}, we reach a contradiction.
\end{proof}

\bibliographystyle{math}
\bibliography{reverse}
\end{document}